\documentclass[11pt,letterpaper,reqno]{amsart}
\usepackage{tikz}
\usetikzlibrary{positioning, shapes.geometric, arrows.meta}
\usepackage{amssymb}
\usepackage{amsmath}
\usepackage{amsthm}
\usepackage{amsfonts}
\IfFileExists{bbm.sty}{\usepackage{bbm}}{}
\providecommand{\mathbbm}[1]{\mathbb{#1}}
\usepackage{enumitem}
\usepackage{pgfplots}
\pgfplotsset{compat=1.18}
\usepackage{booktabs}

\usepackage{graphicx}
\usepackage[T1]{fontenc}
\usepackage{lmodern}
\usepackage{float}
\usepackage{doi}
\usepackage{tikz}
\usepackage{pgfplots}
\pgfplotsset{compat=1.18}
\usetikzlibrary{arrows.meta}

\usepackage{xcolor}
\usepackage{hyperref}
\usepackage{bookmark}

\hypersetup{
	colorlinks=true,
	linkcolor=red,       
	citecolor=green,     
	urlcolor=magenta,    
	filecolor=magenta,
	linktoc=all,
	pdfstartview={FitH},
	pdftitle={Sendov's conjecture holds for all n >= 10^200000},
	pdfauthor={Teng Zhang},
	pdfsubject={Sendov's conjecture holds for all n >= 10^200000},
	pdfkeywords={Sendov's conjecture, critical points, logarithmic potential, balayage}
}

\newtheorem{thm}{Theorem}[section]
\newtheorem{lem}[thm]{Lemma}
\newtheorem{prop}[thm]{Proposition}
\newtheorem{cor}[thm]{Corollary}

\newtheorem{conj}[thm]{Conjecture}
\theoremstyle{definition}

\newtheorem{rem}[thm]{Remark}

\numberwithin{equation}{section}
\numberwithin{table}{section}

\newcommand{\C}{\mathbb C}
\newcommand{\D}{\mathbb D}
\newcommand{\E}{\mathbb E}
\newcommand{\Pp}{\mathbb P}
\newcommand{\1}{\mathbbm 1}
\newcommand{\dd}{\,d}
\newcommand{\supp}{\operatorname{supp}}
\newcommand{\dist}{\operatorname{dist}}
\newcommand{\wind}{\operatorname{wind}}
\newcommand{\Bal}{\operatorname{Bal}}
\newcommand{\Norg}{N_{\mathrm{org}}}
\newcommand{\Nstar}{N_{\ast}}

\makeatother

\begin{document}

\title[Sendov's conjecture holds for all $n\ge 10^{200000}$ ]
{Sendov's conjecture holds for every degree
	$n\ge 10^{200000}$}

\author[T.~Zhang]{Teng Zhang}

\address{School of Mathematics and Statistics, Xi'an Jiaotong University, Xi'an 710049, P. R. China}
\email{teng.zhang@stu.xjtu.edu.cn}



\begin{abstract}
Sendov's conjecture, first formulated in 1958, asserts that if a
complex polynomial \(f\) of degree \(n \geq 2\) has all its zeros in
the closed unit disk $\{z \in \mathbb{C} : |z| \leq 1\},$
then, for every zero \(\lambda_0\) of \(f\), there exists a critical
point \(\zeta\) of \(f\) such that $ |\zeta-\lambda_0| \leq 1.$
The conjecture was previously known to hold for polynomials of degree
\(n \leq 8\), as well as in several special higher-degree cases. In
2022, using compactness methods, balayage, and the argument principle,
Tao~\cite{Tao22} proved that there exists an absolute constant \(n_0\)
such that Sendov's conjecture holds for all \(n \geq n_0\). However,
Tao's argument does not provide an explicit admissible value of
\(n_0\). In the present paper, we make Tao's result effective and prove
that one may take
\[
n_0 = 10^{200000}.
\]
\end{abstract}

\maketitle
\enlargethispage{4pt}

\section{Introduction}
\label{sec:introduction}

For \(z_0\in\C\) and \(r>0\), write
\[
 D(z_0,r)=\{z\in\C:|z-z_0|<r\},
 \qquad
 \overline{D(z_0,r)}=\{z\in\C:|z-z_0|\le r\},
\]
and let \(\D=D(0,1)\).  

In 1958, the Bulgarian mathematician Blagovest Sendov formulated the
conjecture now known as \emph{Sendov's conjecture}. It was subsequently
included in Hayman's influential 1967 monograph
\emph{Research Problems in Function Theory}~\cite[Problem~4.5]{Hay67}.

\begin{conj}[Sendov]
\label{conj:sendov}
Let \(f\) be a complex polynomial of degree \(n\ge2\), all of whose
zeros lie in \(\overline{\D}\).  If \(\lambda_0\) is a zero of \(f\),
then \(f\) has a critical point in \(\overline{D(\lambda_0,1)}\).
\end{conj}

We first give a historical overview of the progress toward
Conjecture~\ref{conj:sendov}.

\begin{itemize}[
	label=\(\bullet\),
	leftmargin=2em,
	itemsep=0.8em,
	topsep=0.5em
	]
	
\item In 1968, Rubinstein~\cite{Rub68} proved
Conjecture~\ref{conj:sendov} when the distinguished zero
\(\lambda_0\) lies on the unit circle. In the same paper, he also
verified the cases \(n=3,4\).
	
\item In 1969, Goodman, Rahman, and Ratti~\cite{GRR69} established a
stronger localization result for a distinguished zero
\(\lambda_0\) on the unit circle, proving that \(f\) has a critical
point \(\zeta\) satisfying
$
\left|\zeta-\lambda_0/2\right|\leq 1/2.
$
In the same year, Joyal~\cite{Joy69} independently verified Conjecture~\ref{conj:sendov} for distinguished zeros on the unit circle; namely, if
\(|\lambda_0|=1\), then there exists a critical point \(\zeta\) of
\(f\) such that
$
|\zeta-\lambda_0|\leq 1.
$
Also in 1969, Meir and Sharma~\cite{MS69} proved Conjecture~\ref{conj:sendov}
for all polynomials of degree \(n\leq 5\).

\item In 1985, Bojanov, Rahman, and Szynal~\cite{BRS85} proved that if
\(f(z)=\prod_{j=1}^{n}(z-\lambda_j)\) has all its zeros in the closed
unit disk, then, for every \(\nu=1,\ldots,n\), the closed disk
$
\overline{D}\!\left(
\lambda_\nu,
\left(1+\left|\prod_{j=1}^{n}\lambda_j\right|\right)^{1/n}
\right)
$
contains at least one zero of \(f'\). Consequently, for each zero
\(\lambda_\nu\) of \(f\), there exists a critical point \(\zeta\) of
\(f\) such that \(|\zeta-\lambda_\nu|\leq2^{1/n}\). Thus, they obtained
Sendov's conclusion with the radius \(1\) replaced by the asymptotically
sharp radius
$
2^{1/n}=1+\frac{\log2}{n}+O(n^{-2}).
$
	
	\item In 1991,
	Brown 	\cite{Bro91} proved the case $n=6$.
	
	\item In 1993, Miller \cite{Mil93} proved that
	for every fixed degree \(n\),  there exists a
	constant \(\varepsilon_n>0\), depending only on \(n\), such that
$
	|\lambda_0|\geq 1-\varepsilon_n
\Longrightarrow
$ there exists a critical point \(\zeta\) of \(f\)
satisfying
$
|\zeta-\lambda_0|\leq 1.
$
	Thus, for each fixed degree $n$, Conjecture~\ref{conj:sendov} holds throughout
	a sufficiently thin annulus adjacent to the unit circle.
	The argument was qualitative in the sense that it did not provide
	a usable explicit value of \(\varepsilon_n\).
	Independently of Miller, V\^aj\^aitu and Zaharescu \cite{VZ93} proved that, for
	every fixed \(n\), there exists \(\varepsilon_n>0\) such that
$
	1-\varepsilon_n\leq|\lambda_0|\leq 1
\Longrightarrow
$ there exists a critical point \(\zeta\) of \(f\)
satisfying
$
|\zeta-\lambda_0|\leq 1.
$
	Their proof likewise did not yield an explicit numerical estimate
	for \(\varepsilon_n\).

	\item In 1996,
	Borcea \cite{Bor96}  proved the case $n\le 7$.

	\item In 1999,
	Brown and Xiang \cite{BX99} proved the case $n\le 8$.

	\item In 2010,
	Chijiwa \cite{Chi10} initiated a quantitative version of the near-boundary
	analysis. Under the hypothetical assumption that a zero
	\(\lambda_0\) is close to the unit circle but has no critical point
	within unit distance, he derived explicit estimates forcing the
	critical points to lie close to the origin and forcing the zeros of
	the polynomial to be close to a rotated configuration of the
	\(n\)-th roots of unity.
	These estimates provided the quantitative foundation for his
	explicit near-boundary theorem in the following year.
	
	\item In 2011,
	Chijiwa \cite{Chi11} made the near-boundary result fully explicit. Let \(n\geq4\)
	and define
$
	\varepsilon_n^{\mathrm C}
	:=\frac{1}{2n^9 4^n}.
$
	If
$
	|\lambda_0|\geq 1-\varepsilon_n^{\mathrm C},
$
	then there exists a critical point \(\zeta\) such that
	\[
	|\zeta-\lambda_0|
	\leq
	1-c_n\bigl(1-|\lambda_0|\bigr),
	\]
	where
	\[
	c_n=
	\begin{cases}
		\dfrac14,
		& n\equiv0\pmod 4,\\[6pt]
		\dfrac{n-3}{4(n-1)},
		& n\equiv1\pmod 4,\\[6pt]
		\dfrac{n-6}{4(n-1)},
		& n\equiv2\pmod 4,\\[6pt]
		\dfrac{n-9}{4(n-1)},
		& n\equiv3\pmod 4.
	\end{cases}
	\]
	In particular, together with the already known low-degree cases,
	this gives Conjecture~\ref{conj:sendov} whenever the distinguished zero is
	exponentially close to the unit circle.
\item In the same year, Bojanov~\cite{Boj11} showed that
Conjecture~\ref{conj:sendov} holds whenever
$
|\lambda_0|\leq \frac{1}{n-1}.
$
We note that the same conclusion follows immediately from the classical
de Bruijn--Springer mean-modulus inequality~\cite{dBS47}.
	
	\item In 2014,
	Kasmalkar \cite{Kas14} enlarged Chijiwa's admissible near-boundary
	region. For \(n\geq8\), define
$
	\varepsilon_n^{\mathrm K}
	:=\frac{90}{n^{12}\log n}.
$
	If
$
	|\lambda_0|\geq 1-\varepsilon_n^{\mathrm K},
$
	then there exists a critical point \(\zeta\) satisfying
	\[
	|\zeta-\lambda_0|
	\leq
	1-c_n\bigl(1-|\lambda_0|\bigr),
	\]
	where \(c_n\) is the same explicit constant as in Chijiwa's theorem.
	Thus Kasmalkar improved the required closeness from the exponentially
	small scale
$
	\frac{1}{2n^9 4^n}
$
	to the polynomial scale
$
	\frac{90}{n^{12}\log n}.
$
In the same year, D\'egot~\cite{Deg14} proved a high-degree result for
each fixed intermediate radius. More precisely, for every fixed
\(r\in(0,1)\), there exists an integer \(N(r)\) such that, whenever
\(\lambda_0\) is a zero of \(f\) with \(|\lambda_0|=r\) and
\(n\geq N(r)\), there exists a critical point \(\zeta\) of \(f\)
satisfying
$
|\zeta-\lambda_0|\leq 1.
$
The threshold in this theorem depends on the modulus of the
distinguished zero and is therefore not uniform for \(r\in(0,1)\).
	
	\item In 2020, 
	Chalebgwa \cite{Cha20} made the dependence in D\'egot's theorem explicit. If
$
	0<r=|\lambda_0|<1
$
	and
$
	n\geq
	\frac{20800}{r^7(1-r)^4},
$
	then
 there exists a critical point \(\zeta\) of \(f\)
satisfying
$
|\zeta-\lambda_0|\leq 1.
$
	Equivalently, for a fixed degree \(n\), this result covers the
	intermediate range in which \(r\) is neither too close to \(0\) nor
	too close to \(1\).
	
	\item In 2022, Tao~\cite{Tao22} proved that there exists an absolute constant \(n_0\)
	such that Conjecture~\ref{conj:sendov} holds for all \(n \geq n_0\). 
	His proof separates the argument into three regimes:
	\begin{enumerate}
		\item when \(|\lambda_0|\) is bounded away from both \(0\) and
		\(1\), it uses the results of D\'egot and Chalebgwa;
		
		\item when \(|\lambda_0|\) is close to \(1\), it refines Miller's
		near-boundary argument and invokes Chijiwa's quantitative theorem
		in the extremely close regime;
		
		\item when \(|\lambda_0|\) is close to \(0\), it introduces a new
		argument based on compactness, balayage, logarithmic potentials,
		and the argument principle.
	\end{enumerate}
	Because the proof uses a compactness-and-contradiction argument, it
	does not yield an explicit numerical value of \(n_0\).
	
\end{itemize}

In this paper, we make Tao's theorem for sufficiently large
degrees~\cite{Tao22} effective. In Tao's theorem, one may take
\[
n_0=10^{200000}.
\]

\begin{thm}
	\label{thm:main}
	Sendov's conjecture holds for every degree
	\[
	n\geq \Nstar:=10^{200000}.
	\]
\end{thm}

For each fixed degree \(n\), Sendov's conjecture can be formulated as a
first-order sentence in the language of ordered fields. Hence, by
Tarski's quantifier-elimination theorem~\cite{Tar31} for real closed
fields, it is decidable in finite time. An explicit real-algebraic
formulation, together with accompanying computational implementations,
was given by Spjut~\cite[Chapters~3--4]{Spj20}.

\begin{rem}
	Theorem~\ref{thm:main} implies that Sendov's conjecture is effectively
	decidable. Indeed, it remains only to verify the finitely many degrees
$
	2\leq n<\Nstar,
$
	and, for each such degree, the validity of the conjecture can be decided
	in finite time by quantifier elimination. Thus, applying the decision
	procedure successively to these finitely many degrees yields, in
	principle, a finite algorithm for deciding Sendov's conjecture in full.
\end{rem}

\begin{rem}
	The numerical bound in Theorem~\ref{thm:main} can clearly be improved by refining the constants and
	parameter choices throughout the proof.  Such an optimization is not
	essential to the present work; the main point is to obtain a completely
	explicit universal degree threshold.
\end{rem}

\vspace{0.1in}
\noindent\textbf{Sketch of the proofs and new ideas.}
The magnitude of \(\Nstar\) in Theorem~\ref{thm:main} reflects the deliberate use of broad
numerical margins.  The logical structure is nonetheless short.  After
rotation and multiplication by a nonzero scalar, a hypothetical
counterexample has a monic polynomial \(f\), a distinguished zero
\(a\in[0,1]\), and no critical point in
\(\overline{D(a,1)}\).  Four ranges cover all \(a>0\):
\[
\begin{array}{ccl}
0<a\le n^{-1/8}
&:& \text{the effective near-origin argument of Section~\ref{sec:origin}},\\
a>n^{-1/8},\quad 1-a\ge n^{-1/64}
&:& \text{Chalebgwa's explicit theorem},\\
a>n^{-1/8},\quad 8^{-n}\le1-a<n^{-1/64}
&:& \text{the effective boundary argument of Section~\ref{sec:circle}},\\
a>n^{-1/8},\quad 0\le1-a<8^{-n}
&:& \text{Chijiwa's explicit theorem}.
\end{array}
\]
The case \(a=0\) follows immediately from the Gauss--Lucas theorem \cite[pp.~71--72, Theorem~2.1.1]{RS02}.
The proof of Theorem~\ref{thm:main} is assembled in
Section~\ref{sec:assembly}.

We now describe the two new quantitative ingredients. Near the origin,
under the counterexample hypothesis in the regime \(a\to0\),
Tao~\cite[Theorem~1.10(i)]{Tao22} proved that the subsequential
limiting laws of the zeros and critical points are both supported on
the left semicircle \(C\), and that these two limiting laws coincide.
Instead of passing to a limiting measure, we compare the balayages of
the zero and critical-point measures at two exterior radii. A
Poisson-kernel test then places all but an explicitly bounded
proportion of the zeros near the left semicircle. A fixed
winding-number dichotomy and a selected circle complete the argument.

Near the unit circle, Tao's compactness argument \cite[Proposition~5.3]{Tao22} produces an at most
countable exceptional set.  We replace it by fewer than
\(10^{32}\) explicit obstacles.  Away from a set of angles of total
length less than \(10^{-20}\),
an actual zero close to the prescribed boundary direction can be joined
to the small disk by a path that stays a definite distance from every
critical point.  Integration of \(f''/f'\) along this path gives an
effective level-set relation.  The zeroth and second Fourier coefficients
of the boundary potential then contradict the geometry of the arc
$
 \{z:|z|\le1,\ |1-z|=1\}.
$

\vspace{0.1in}
\noindent\textbf{Organization of the paper.}
Section~\ref{sec:preliminaries} collects the basic identities and
support properties for the zero and critical-point measures, and
establishes an explicit comparison between their moments.
Section~\ref{sec:origin} develops an effective argument for
distinguished zeros near the origin, based on quantitative balayage
estimates, harmonic-measure separation, and a winding-number
argument.  Section~\ref{sec:circle} treats distinguished zeros near
the unit circle.  Its main ingredients are a finite
obstacle-avoidance construction, an effective level-set relation for
the critical-point potential, and incompatible Fourier and geometric
estimates for the second moment.  In
Section~\ref{sec:assembly}, the two endpoint arguments are combined
with the quantitative results of Chalebgwa and Chijiwa to prove
Theorem~\ref{thm:main}.  Finally,
Appendix~\ref{sec:numerical-audit} records the principal numerical
inequalities and margins used throughout the proof.

\vspace{0.1in}
\noindent\textbf{Acknowledgments.} The author would like to express his sincere gratitude to Professor Minghua Lin for introducing him to Sendov's conjecture in 2021 and for teaching him how matrix methods can be used to study the geometry of polynomials. He is also deeply grateful to Professor Zongben Xu for his kind support and concern regarding both his academic work and personal life. The author would also like to thank Professors Stephen Drury, Hristo Sendov, and Vilmos Totik for the valuable discussions and exchanges during his PhD years. Special thanks are due to Professor Lajos Molnár for his generous assistance and hospitality during the author’s visit to the University of Szeged in 2026. DeepSeek was used for language editing and to assist with literature searches, while MATLAB was used for numerical computations.
This work was supported by the China Scholarship Council, the Young Elite Scientists Sponsorship Program for PhD Students of the China Association for Science and Technology, and the Fundamental Research Funds for the Central Universities at Xi’an Jiaotong University (Grant No. xzy022024045).

\section{Preliminaries}
\label{sec:preliminaries}

Let \(n\ge2\), and let \(f\) be monic of degree \(n\), all of whose
zeros lie in \(\overline{\D}\). Let \(\lambda_1,\ldots,\lambda_n\) and
\(\xi_1,\ldots,\xi_{n-1}\) denote the zeros of \(f\) and \(f'\),
respectively, counted with multiplicity.  We write
\(\lambda\) and \(\xi\) for uniformly chosen elements of these two
multisets.  Define
\begin{equation*}
 U_X(z):=\E\log\frac1{|z-X|},
 \qquad
 s_X(z):=\E\frac1{z-X}.
\end{equation*}
These functions are used only where the displayed expressions are
finite.

\begin{lem}[{\cite[Lemma~1.6, (ii)--(v)]{Tao22}}]
\label{lem:basic-identities}
For every monic polynomial \(f\) as above,
\begin{equation*}
 \E\lambda=\E\xi.
\end{equation*}
Moreover,
\begin{align}
 U_\lambda(z)&=-\frac1n\log|f(z)|, \label{eq:root-potential}\\
 U_\xi(z)&=\frac{\log n-\log|f'(z)|}{n-1}, \label{eq:critical-potential}\\
 s_\lambda(z)&=\frac{f'(z)}{nf(z)}, \label{eq:log-derivative}\\
 s_\xi(z)&=\frac{f''(z)}{(n-1)f'(z)}. \label{eq:critical-log-derivative}
\end{align}
Here \eqref{eq:root-potential} and \eqref{eq:log-derivative}
hold whenever \(f(z)\ne0\), whereas
\eqref{eq:critical-potential} and
\eqref{eq:critical-log-derivative} hold whenever \(f'(z)\ne0\).
If \(f(z)f'(z)\ne0\), then
\begin{equation}
 U_\lambda(z)-\frac{n-1}{n}U_\xi(z)
 =\frac1n\log|s_\lambda(z)|.
 \label{eq:potential-identity}
\end{equation}
\end{lem}

The next elementary observation records the support restrictions imposed
by a counterexample.

\begin{lem}[{\cite[Lemma~1.6, (i), (ii)]{Tao22}}]
\label{lem:counterexample-support}
Suppose that \(f(a)=0\), where \(a\in[0,1]\), and that \(f'\) has no
zero in \(\overline{D(a,1)}\).  Then
\[
 |\xi|\le1,\qquad |a-\xi|>1
\]
almost surely, and \(\E\lambda=\E\xi\).
\end{lem}

The following observation, due to Komarova and
Rivin~\cite[Lemma~5.7]{KR03}, plays an important role in our
derivation.  It has also recently been used by Tang and the author in
the study of Schoenberg-type inequalities~\cite{TZ25}.

\begin{lem}[{\cite[Lemma~5.7]{KR03}}]\label{lem:zpprimez}
	Let \( n \geq 2 \) and \( \lambda_1, \lambda_2, \ldots, \lambda_n \in \mathbb{C} \). Let \( p(z) = \prod_{j=1}^n (z - \lambda_j) \), and define \( D = \operatorname{diag}(\lambda_1, \lambda_2, \ldots, \lambda_n) \). Let \(I_n\) denote the \(n\times n\) identity matrix, and let \(J_n\) denote the \(n\times n\) matrix whose entries are all equal to \(1\). Then the characteristic polynomial of the matrix \( D \left( I_n - \frac{1}{n} J_n \right) \) has the same multiset of zeros as the polynomial \( z p'(z) \). 
\end{lem}

\begin{rem}\label{rem:zpprimez}
	Clearly, \(I_n - \tfrac{1}{n} J_n\) is a projection matrix. Since \(AB\) and \(BA\) have the same spectrum for any \(n\times n\) complex matrices \(A,B\), all eigenvalues of
	\(D \left(I_n - \tfrac{1}{n} J_n\right)\) coincide with those of \( \left(I_n - \tfrac{1}{n} J_n\right) D \left(I_n - \tfrac{1}{n} J_n\right)\).
\end{rem}

Tao's exterior-potential argument~\cite[p.~363]{Tao22} gives the non-asymptotic moment
comparison
\[
\left|\E\lambda^m-\E\xi^m\right|
=O\left(\frac{m\log m}{n}\right),
\qquad m\ge2.
\]
Lemma~\ref{lem:moment-comparison} below removes the logarithmic
loss and makes the constant explicit, giving \(5m/n\) for every
\(m\ge1\).  Its proof uses the $D$-companion matrix method rather than Tao's potential-theoretic argument.

\begin{lem}
\label{lem:moment-comparison}
For every integer \(m\ge1\),
\begin{equation}
 \left|\E\lambda^m-\E\xi^m\right|\le\frac{5m}{n}.
 \label{eq:moment-comparison}
\end{equation}
\end{lem}

\begin{proof}
Let \(D=\operatorname{diag}(\lambda_1,\ldots,\lambda_n)\) and \(P=I_n-\frac{1}{n}J_n\).  By Lemma \ref{lem:zpprimez} and Remark \ref{rem:zpprimez},
$PDP$ has characteristic polynomial \(zf'(z)/n\).  
Put \(Q=I_n-P=\frac{1}{n}J_n\).  Since \(Q\) has rank one and \(P,Q\)
are orthogonal projections,
\[
D-PDP=QD+PDQ.
\]
Consequently,
\[
\operatorname{rank}(D-PDP)\le
\operatorname{rank}(QD)+\operatorname{rank}(PDQ)\le2.
\]
Moreover, $\|D\|=\max_j|\lambda_j|\le 1$, where $\|\cdot\|$ denotes the spectral norm. Therefore
\[
\|D-PDP\|
\le\|QD\|+\|PDQ\|
\le2\|D\|\le2.
\]
The telescoping identity for matrix powers 
\[
 D^m-(PDP)^m
 =\sum_{k=0}^{m-1}D^{m-1-k}(D-PDP)(PDP)^k
\]
and \(|\operatorname{tr} A|\le\operatorname{rank}(A)\|A\|\) give
\[
 \left|\operatorname{tr}D^m-\operatorname{tr}(PDP)^m\right|\le4m.
\]
Since
\[
 \operatorname{tr}\bigl((PDP)^m\bigr)=\sum_{j=1}^{n-1}\xi_j^m
 \qquad\text{and}\qquad
 \left|\operatorname{tr}\bigl((PDP)^m\bigr)\right|\le n-1,
\]
we obtain
\[
\begin{aligned}
 \left|\E\lambda^m-\E\xi^m\right|&\le
 \frac1n\left|
 \operatorname{tr}(D^m)
 -\operatorname{tr}\bigl((PDP)^m\bigr)\right|+
 \left|\frac1n-\frac1{n-1}\right|
 \left|\operatorname{tr}\bigl((PDP)^m\bigr)\right|\\
 &\le\frac{4m+1}{n}\le\frac{5m}{n}.
\end{aligned}
\]
This proves \eqref{eq:moment-comparison}.
\end{proof}

The preceding lemmas will be invoked in both endpoint regimes.  We turn
first to the origin, where a direct winding-number argument is available.

\section{An effective argument near the origin}
\label{sec:origin}

Throughout this section, \(f\) is monic of degree \(n\), all of its zeros
lie in \(\overline{\D}\), and \(a\in(0,1)\) is a zero of \(f\).  We
argue by contradiction and assume that
\begin{equation}
 f'(\zeta)\ne0\qquad\text{whenever}\qquad|\zeta-a|\le1.
 \label{eq:org-contrary}
\end{equation}
The notation \(\lambda,\xi,U_X,s_X\) is as in
Section~\ref{sec:preliminaries}.  Thus
\begin{equation}
 s_\lambda(z)=\frac{f'(z)}{nf(z)},
 \label{eq:org-stieltjes}
\end{equation}
and Lemma~\ref{lem:counterexample-support} places \(\xi\) in the lune
\begin{equation*}
 \mathcal L_a
 :=\{w\in\overline{\D}:|w-a|>1\}
\end{equation*}
almost surely.

Tao~\cite[Theorem~3.1]{Tao22} proved that there is no sequence of
counterexamples of degrees \(n_k\to\infty\) whose distinguished zeros
\(a_k\) satisfy
\[
a_k=o\left(\frac1{\log n_k}\right).
\]
That result is
asymptotic: its proof passes to a limiting root measure and selects an
annulus from the zero set of the limiting Stieltjes transform, and
therefore does not produce a numerical degree threshold.  

The main theorem of this section is an effective specialization of Tao's result.  It replaces the
asymptotic hypothesis by the finite condition \(0<a\le n^{-1/8}\) and
supplies an explicit threshold \(n\ge\Norg\).  Since
$
n^{-1/8}=o\left(\frac1{\log n}\right),
$
this does not enlarge Tao's asymptotic range; the new content is the
uniform finite-degree argument with completely explicit constants.  In
particular, the nonconstructive limiting annulus is replaced below by a
quantitative concentration estimate and a fixed winding-number
dichotomy.

\begin{thm}
\label{thm:org-effective}
Set
\begin{equation}
 \Norg:=
 \bigl(97644800000000000000000\bigr)^{16}.
 \label{eq:org-threshold}
\end{equation}
If \(n\ge\Norg\) and \(0<a\le n^{-1/8}\), then \(f'\) has a zero in
\(\overline{D(a,1)}\).
\end{thm}

We begin its proof by making the comparison of balayages quantitative.
For \(R>1\), \(w\in\overline{\D}\), and \(\theta\in\mathbb R\), put
\begin{equation*}
 P_w^R(Re^{i\theta})
 :=\frac{R^2-|w|^2}{|Re^{i\theta}-w|^2}.
\end{equation*}
The normalization is
\[
 \int_{-\pi}^{\pi}P_w^R(Re^{i\theta})\,\frac{\dd\theta}{2\pi}=1.
\]
For a random variable \(X\) taking values in \(\overline{\D}\), define
\begin{equation*}
 \Bal_R(X)(Re^{i\theta})
 :=\E P_X^R(Re^{i\theta}).
\end{equation*}
This is the density, with respect to normalized arclength measure on
\(\partial D(0,R)\), of the classical balayage of the law of \(X\)
onto \(\partial D(0,R)\); see~\cite{Gus04}.

Tao  \cite[p.~363, (2.9)]{Tao22} established the non-asymptotic
balayage estimate
\[
\left\|\Bal_R(\lambda)-\Bal_R(\xi)\right\|_
{L^\infty(\partial D(0,R))}
=
O\left(
\frac{\log(1/(R-1))}{n(R-1)^2}
\right),
\qquad 1<R\le\frac32.
\]
  The following lemma is a fully
explicit version of this estimate.  With \(R=1+t\), it gives the
constant \(72\) while preserving the same dependence on \(n\) and
\(t\).  Its proof follows Tao's balayage--Fourier argument, with all
implicit constants tracked quantitatively for later use.

\begin{lem}
\label{lem:org-effective-balayage}
Let \(0<t\le1/2\) and \(R=1+t\).  Then
\begin{equation*}
 \left\|\Bal_R(\lambda)-\Bal_R(\xi)\right\|_
 {L^\infty(\partial D(0,R))}
 \le\frac{72\log(3/t)}{nt^2}.
\end{equation*}
\end{lem}

\begin{proof}
Set \(S=1+t/2\).  If \(z=Se^{i\theta}\), then
\[
 |s_\lambda(z)|\le\frac1{S-1}
\]
and
\[
 \Re\!\left(e^{i\theta}s_\lambda(Se^{i\theta})\right)
 =\E\Re\frac1{S-\lambda e^{-i\theta}}
 \ge\frac1{S+1}.
\]
Consequently,
\[
 \frac1{S+1}\le|s_\lambda(z)|\le\frac1{S-1}.
\]
Likewise,
\[
 -\log(S+1)\le U_\xi(z)\le-\log(S-1).
\]
Using \eqref{eq:potential-identity} and \(S-1=t/2\), we obtain
\begin{equation}
 \|U_\lambda-U_\xi\|_{L^\infty(\partial D(0,S))}
 \le\frac{2\log(3/t)}n.
 \label{eq:org-potential-uniform}
\end{equation}

For \(|w|\le1\), the absolutely convergent expansion
\[
 \log\frac1{|Se^{i\theta}-w|}
 =-\log S+
 \sum_{m=1}^{\infty}\frac{\Re(w^me^{-im\theta})}{mS^m}
\]
shows, on extracting the \(m\)-th Fourier coefficient in
\eqref{eq:org-potential-uniform}, that
\begin{equation*}
 \left|\E\lambda^m-\E\xi^m\right|
 \le\frac{4mS^m\log(3/t)}n
 \qquad(m\ge1).
\end{equation*}
On the other hand,
\[
 \Bal_R(X)(Re^{i\theta})
 =1+2\Re\sum_{m=1}^{\infty}
 \frac{\E X^m}{R^m}e^{-im\theta}.
\]
It follows that
\[
 \|\Bal_R(\lambda)-\Bal_R(\xi)\|_\infty
 \le\frac{8\log(3/t)}n
 \sum_{m=1}^{\infty}m\left(\frac SR\right)^m.
\]
Since
$
 1-\frac SR=\frac{t}{2(1+t)}
$
and
\[
 \sum_{m=1}^{\infty}m\left(\frac SR\right)^m
 =\frac{S/R}{(1-S/R)^2}
 \le\frac{4(1+t)^2}{t^2}
 \le\frac9{t^2},
\]
the claimed bound follows.
\end{proof}

We next quantify the differing harmonic-measure behavior of the left
semicircle and of points separated from it.  Put
\begin{equation*}
	C:=\left\{e^{i\phi}:\frac{\pi}{2}\le\phi\le\frac{3\pi}{2}\right\},
\end{equation*}
and, for \(0<\rho\le10^{-2}\), let
\begin{equation*}
	I_\rho:=
	\left[-\frac{\pi}{2}+\frac{\rho}{4},
	\frac{\pi}{2}-\frac{\rho}{4}\right].
\end{equation*}
We abbreviate
\[
\omega_R(w,I_\rho)
:=\int_{I_\rho}P_w^R(Re^{i\theta})\,\frac{\dd\theta}{2\pi}.
\]

In the proof of \cite[Theorem~1.10(i)]{Tao22}, Tao   used a harmonic-measure
separation between points bounded away from \(C\) and points in the
thin lune \(\mathcal L_a\).  More precisely, for a fixed compact set
\(K\subset\overline{\D}\setminus C\) and a suitable fixed arc \(I\),
the harmonic measure of \(I\) is uniformly bounded below for starting
points in \(K\), whereas for starting points in \(\mathcal L_a\) it is
\(O(a+R-1)\).  The following lemma
makes the dependence on the separation scale \(\rho\) explicit.  With
\(R=1+t\) and the arc \(I_\rho\), it gives the respective bounds
\[
\frac{\rho}{64}
\qquad\text{and}\qquad
\frac{250(a+t)}{\rho^2}.
\]
Thus the underlying harmonic-measure argument is the same as Tao's,
while all parameter dependence and numerical constants are made
explicit.

\begin{lem}
\label{lem:org-poisson-test} Let \(0<\rho\le10^{-2}\), and suppose that
\[
 0<t\le\frac{\rho}{4},\qquad
 0<a\le\frac{\rho}{10},\qquad R=1+t.
\]
Then the following estimates hold.
\begin{enumerate}[label=\textup{(\roman*)}]
\item If \(w\in\overline{\D}\) and \(\dist(w,C)\ge\rho\), then
\begin{equation}
 \omega_R(w,I_\rho)\ge\frac{\rho}{64}.
 \label{eq:org-poisson-lower}
\end{equation}
\item If \(w\in\mathcal L_a\), then
\begin{equation}
 \omega_R(w,I_\rho)\le\frac{250(a+t)}{\rho^2}.
 \label{eq:org-poisson-upper}
\end{equation}
\end{enumerate}
\end{lem}

\begin{proof}
Write \(w=re^{i\phi}\).  If \(r\le1-\rho/2\), then
\[
 R^2-r^2\ge1-(1-\rho/2)^2\ge\frac{3\rho}{4},
 \qquad
 |Re^{i\theta}-w|^2<(R+r)^2<5.
\]
As \(|I_\rho|/(2\pi)>0.49\), this already proves
\eqref{eq:org-poisson-lower}.

Assume now that \(r>1-\rho/2\).  The separation from \(C\) forces
\(\phi\in(-\pi/2,\pi/2)\): otherwise \(e^{i\phi}\in C\) and
\(|w-e^{i\phi}|=1-r<\rho\).  Moreover,
\[
 \delta_\phi:=
 \min\left\{\phi+\frac{\pi}{2},\frac{\pi}{2}-\phi\right\}
 >\frac{\rho}{2};
\]
otherwise the nearer endpoint of \(C\) would be at distance at most
\((1-r)+\delta_\phi<\rho\) from \(w\).  Hence
\[
 [\phi-\rho/4,\phi+\rho/4]\subset I_\rho.
\]
Put \(s=R-r\) and
$
 \eta=\min\left\{\frac{s}{R},\frac{\rho}{4}\right\}.
$
If \(|\theta-\phi|\le\eta\), then, using
\(1-\cos u\le u^2/2\), \(r\le R\), and
\(|\theta-\phi|\le s/R\), we have
\[
 |Re^{i\theta}-re^{i\phi}|^2
 =s^2+2Rr(1-\cos(\theta-\phi))\le2s^2
\]
and \(R^2-r^2=s(R+r)\ge sR\).  Thus the Poisson kernel is at least
\(R/(2s)\).  When \(\eta=s/R\), integration gives \(1/(2\pi)\).
When \(\eta=\rho/4\), the inequality
\[
 s<R-(1-\rho/2)\le3\rho/4
\]
gives instead \(1/(6\pi)\).  Both lower bounds are larger than
\(\rho/64\).

For the upper bound, write \(w=x+iy\).  The inequalities
\(|w|\le1\) and \(|w-a|>1\) imply
\[
 x<\frac a2,\qquad
 1-|w|^2<a^2-2ax\le a^2+2a\le3a.
\]
Therefore \(R^2-|w|^2\le3(a+t)\).  For \(\theta\in I_\rho\),
\[
 R\cos\theta-x
 \ge\sin(\rho/4)-\frac a2
 \ge\frac{\rho}{5}-\frac{\rho}{20}
 =\frac{3\rho}{20}.
\]
It follows that
\[
 P_w^R(Re^{i\theta})
 \le\frac{3(a+t)}{9\rho^2/400}
 <\frac{250(a+t)}{\rho^2}.
\]
Integration proves \eqref{eq:org-poisson-upper}.
\end{proof}

The quantitative part of \cite[Theorem~1.10(i)]{Tao22} states that,
for every fixed compact set
\[
K\subset\overline{\D}\setminus C,
\]
one has
\[
\Pp(\lambda\in K)
\ll_K a+\frac{\log n}{n^{1/3}}.
\]
Thus, for each fixed \(\rho>0\), Tao's estimate already applies to
\[
K_\rho:=
\{w\in\overline{\D}:\dist(w,C)\ge\rho\}.
\]
However, the implied constant in Tao's estimate is not made explicit
as a function of \(\rho\), and therefore cannot be used directly in
the present finite-degree argument. Using the same balayage and
harmonic-measure mechanism, the preceding lemmas give the fully
explicit estimate
\[
\Pp(\lambda\in K_\rho)
\le73000\rho^{-3}
\left(a+n^{-1/3}\log(3n)\right).
\]
The point of the proposition below is therefore the explicit
\(\rho^{-3}\)-dependence and the resulting effective numerical bound
needed for the subsequent winding argument.
\begin{prop}
\label{prop:org-concentration}
Let
$
 \rho=10^{-3}, t=n^{-1/3}$ and  $R=1+t$.
Assume \(t\le\rho/4\) and \(a\le\rho/10\).  If
$
 q:=\Pp\bigl(\dist(\lambda,C)\ge\rho\bigr),
$
then
\begin{equation*}
 q\le73000\rho^{-3}
 \left(a+n^{-1/3}\log(3n)\right).
\end{equation*}
\end{prop}

\begin{proof}
Lemma~\ref{lem:org-poisson-test} gives
\[
 \frac{\rho q}{64}\le\E\omega_R(\lambda,I_\rho).
\]
Lemmas~\ref{lem:counterexample-support},
\ref{lem:org-poisson-test}, and \ref{lem:org-effective-balayage} give
\[
 \E\omega_R(\lambda,I_\rho)
 \le\frac{250(a+t)}{\rho^2}
 +\frac{72\log(3/t)}{nt^2}.
\]
Since \(t=n^{-1/3}\), this implies
\[
 q\le16000\rho^{-3}(a+t)
 +4608\rho^{-1}t\log(3/t).
\]
Now use \(\rho<1\), \(\log(3/t)\le\log(3n)\), and
\(\log(3n)>1\), and enlarge the coefficient to \(73000\).
\end{proof}

For the remainder of this section, fix
\[
\rho:=10^{-3},
\qquad
q:=\Pp\bigl(\dist(\lambda,C)\ge\rho\bigr).
\]
We record the numerical consequence that will be used below.  Put
\(x=n^{1/16}\).  If \(n\ge\Norg\), then
\[
 x\ge97644800000000000000000>10^{22},
 \qquad
 \log n\le x,\qquad \log(3n)\le2x.
\]
For \(a\le n^{-1/8}\), Proposition~\ref{prop:org-concentration}
therefore gives
\begin{align}
 q\log n
 &\le7.3\cdot10^{13}
 \left(n^{-1/8}+n^{-1/3}\log(3n)\right)\log n \notag\\
 &\le7.3\cdot10^{13}
 \left(x^{-1}+2x^{-10/3}\right)
 <\frac1{4400000}.
 \label{eq:org-q-log-bound}
\end{align}
In particular, \(q<1\).

The remaining analytic input is a uniform winding alternative for
measures on \(C\).

\begin{lem}
\label{lem:org-cauchy-winding}
Let \(\nu\) be a probability measure supported on \(C\), let \(W\) have
law \(\nu\), and put
\[
 \alpha=\E\overline W,\qquad
 s_\nu(z)=\E\frac1{z-W}.
\]
Then
\begin{equation}
 s_\nu(z)=-\alpha+z+\mathcal R_\nu(z),
 \label{eq:org-cauchy-expansion}
\end{equation}
where, for \(|z|<1\),
\begin{equation}
 |\mathcal R_\nu(z)|
 \le2|\alpha||z|+\frac{|z|^2}{1-|z|}.
 \label{eq:org-cauchy-remainder}
\end{equation}
If \(|\alpha|\ge1/20\), then for every
\(r\in[1/200,1/100]\),
\begin{equation*}
 \min_{|z|=r}|s_\nu(z)|>0.0388,
 \qquad
 \wind(s_\nu(\partial D(0,r)),0)=0.
\end{equation*}
If \(|\alpha|<1/20\), then for every
\(r\in[1/5,21/100]\),
\begin{equation*}
 \min_{|z|=r}|s_\nu(z)|>0.08,
 \qquad
 \wind(s_\nu(\partial D(0,r)),0)=1.
\end{equation*}
\end{lem}

\begin{proof}
Since \(|W|=1\),
\[
 \frac1{z-W}
 =-\sum_{k=0}^{\infty}z^k\overline W^{\,k+1}.
\]
For \(W=e^{i\phi}\in C\),
\[
 |1+\overline W^{\,2}|=2|\cos\phi|=-2\cos\phi.
\]
Thus
\[
 \E|1+\overline W^{\,2}|
 =-2\Re\alpha\le2|\alpha|.
\]
Adding and subtracting \(z\) in the power series proves
\eqref{eq:org-cauchy-expansion} and
\eqref{eq:org-cauchy-remainder}.

If \(|\alpha|\ge1/20\) and \(r\le1/100\), then
\[
 |s_\nu(z)|
 \ge|\alpha|(1-2r)-r-\frac{r^2}{1-r}
 \ge\frac1{20}\cdot0.98-0.01-\frac{0.01^2}{0.99}
 >0.0388.
\]
The same estimate gives \(|s_\nu(z)+\alpha|<|\alpha|\), so the
straight-line homotopy to \(-\alpha\) avoids zero and has winding
number \(0\).

If \(|\alpha|<1/20\) and \(1/5\le r\le21/100\), then
\[
 |s_\nu(z)-z|
 \le\frac1{20}(1+2r)+\frac{r^2}{1-r}<r.
\]
Indeed, the function
$
 r-\frac1{20}(1+2r)-\frac{r^2}{1-r}
$
has derivative
\[
 \frac9{10}-\frac{r(2-r)}{(1-r)^2}>0
\]
on \([1/5,21/100]\), and it equals \(0.08\) at \(r=1/5\).
Thus the margin is at least \(0.08\) on the stated interval.  Rouch\'e's
theorem compares \(s_\nu\) with \(z\), giving winding number \(1\).
\end{proof}
It remains to choose a boundary circle that avoids the zeros of \(f\)
and on which the exceptional zeros contribute little.  Let
\[
G:=\{\dist(\lambda,C)<\rho\}.
\]
By \eqref{eq:org-q-log-bound}, \(q<1\), so
\(\Pp(G)=1-q>0\).  Choose a Borel nearest-point projection
\(\pi:G\to C\), so that
\(|\lambda-\pi(\lambda)|<\rho\) on \(G\), and let \(\nu\) be the
conditional law of \(\pi(\lambda)\) given \(G\).
  The exact
decomposition is
\begin{align*}
 s_\lambda(z)
 &=(1-q)s_\nu(z)+E_{\mathrm{proj}}(z)+E_{\mathrm{bad}}(z),
 \\
 E_{\mathrm{proj}}(z)
 &:=
 \E\left[\1_G\left(\frac1{z-\lambda}
                   -\frac1{z-\pi(\lambda)}\right)\right],
 \\
 E_{\mathrm{bad}}(z)
 &:=
 \E\frac{\1_{G^c}}{z-\lambda}.
\end{align*}

The radius-selection mechanism used below already appears, in
non-explicit form, in the  proof of \cite[Theorem~3.1]{Tao22}.  There Tao splits off the exceptional
zeros, truncates the singular kernel at the scale \(n^{-10}\), and
averages over an interval of radii.  The Fubini--Tonelli theorem and
Markov's inequality then show that the exceptional contribution is
small for most radii, after which small intervals around the moduli of
the zeros of \(f\) are deleted to ensure that the selected circle
contains no zero.

The following lemma applies the same selection principle to the
decomposition above and makes every estimate explicit.  The
exceptional contribution \(E_{\mathrm{bad}}\) is bounded by
\(8800q\log n\), while the additional error
\(E_{\mathrm{proj}}\), arising from projection onto \(C\), is bounded
by \(0.00161\).  The new idea is the explicit finite-degree
form and the accompanying projection estimate, rather than the
underlying Fubini--Markov radius-selection argument.

\begin{lem}
\label{lem:org-radius-selection}
Assume \(n\ge3\).  Let \(J\) be either
\[
 J_0=[1/200,1/100]
 \qquad\text{or}\qquad
 J_1=[1/5,21/100].
\]
There is \(r\in J\) such that \(f\) has no zero on
\(\partial D(0,r)\) and
\begin{equation*}
 \sup_{|z|=r}|E_{\mathrm{bad}}(z)|
 \le8800q\log n.
\end{equation*}
For this \(r\),
\begin{equation*}
 \sup_{|z|=r}|E_{\mathrm{proj}}(z)|<0.00161.
\end{equation*}
\end{lem}

\begin{proof}
Put \(\delta=n^{-10}\) and
\[
 X(r):=\E\frac{\1_{G^c}}
 {\max\{|r-|\lambda||,\delta\}}.
\]
For every \(s\in[0,1]\) and for either interval \(J\),
\[
 \int_J\frac{\dd r}{\max\{|r-s|,\delta\}}
 \le2+2\log(1/\delta)\le22\log n.
\]
Fubini's theorem gives
\[
 \int_JX(r)\,\dd r\le22q\log n.
\]
Writing \(\ell=|J|\ge1/200\), the set on which
\[
 X(r)>\frac{44q\log n}{\ell}
\]
has measure at most \(\ell/2\).  Delete also the intervals of radius
\(\delta\) centered at the moduli of the \(n\) zeros of \(f\).
Their total length is at most \(2n\delta=2n^{-9}<\ell/2\).  A point
\(r\) remains.  At that point \(f\) has no boundary zero and
\[
 \sup_{|z|=r}|E_{\mathrm{bad}}(z)|
 \le X(r)\le\frac{44q\log n}{\ell}\le8800q\log n.
\]

Finally, \(r\le0.21\), \(\rho=0.001\), and
\(|\lambda-\pi(\lambda)|<\rho\) on \(G\), so
\[
 |E_{\mathrm{proj}}(z)|
 \le\frac{\rho}{(1-r-\rho)(1-r)}
 \le\frac{0.001}{0.789\cdot0.79}<0.00161.
\]
\end{proof}

All ingredients are now in place to prove the endpoint theorem.

\begin{proof}[Proof of Theorem~\ref{thm:org-effective}]
Assume \eqref{eq:org-contrary}.  Set
\(\rho=10^{-3}\) and \(t=n^{-1/3}\).  The hypotheses imply
\[
 t\le\rho/4,\qquad a\le\rho/10,
\]
and \eqref{eq:org-q-log-bound} gives
\begin{equation*}
 8800q\log n\le\frac1{500}.
\end{equation*}

If \(|\alpha|\ge1/20\), use \(J_0\) in
Lemma~\ref{lem:org-radius-selection}; otherwise use \(J_1\).
On the selected circle,
\[
 |s_\lambda(z)-(1-q)s_\nu(z)|
 <0.00161+0.002=0.00361.
\]
By Lemma~\ref{lem:org-cauchy-winding},
\[
 |(1-q)s_\nu(z)|
 >
 \begin{cases}
 (1-q)0.0388,&|\alpha|\ge1/20,\\
 (1-q)0.08,&|\alpha|<1/20.
 \end{cases}.
\]
Since \(q<1/4400000\), both lower bounds exceed \(0.00361\).
The straight-line homotopy between \(s_\lambda\) and
\((1-q)s_\nu\) consequently avoids zero, and
\begin{equation}
 \wind(s_\lambda(\partial D(0,r)),0)\in\{0,1\}.
 \label{eq:org-positive-winding}
\end{equation}

The selected radius satisfies \(r\le0.21\), whereas
\(a\le n^{-1/8}<10^{-3}\).  Hence
\[
\overline{D(0,r)}\subset D(a,1),
\]
because \(r+a<0.211<1\).  By \eqref{eq:org-contrary}, \(f'\)
has no zero in \(\overline{D(0,r)}\), and therefore
\[
N_{D(0,r)}(f')=0.
\]
The construction of \(r\) also ensures that \(f\) has no zero on
\(\partial D(0,r)\), while \(a<1/200\le r\) places at least one zero
of \(f\) inside. Writing \(N_\Omega(g)\) for the number of zeros of
\(g\) in \(\Omega\), counted with multiplicity, the argument
principle applied to \eqref{eq:org-stieltjes} gives
\[
\wind(s_\lambda(\partial D(0,r)),0)
=N_{D(0,r)}(f')-N_{D(0,r)}(f)\le-1,
\]
contradicting \eqref{eq:org-positive-winding}.
\end{proof}

The argument near the origin is thus fully effective.  The boundary
endpoint requires additional geometry because the limiting
critical-point support may have finitely many macroscopic outliers.

\section{An effective argument near the unit circle}
\label{sec:circle}

Let \(n\ge\Nstar\), and suppose again that \(f\) is monic, all its zeros
lie in \(\overline{\D}\), \(f(a)=0\), and
\begin{equation}
 f'(\zeta)\ne0\qquad(|\zeta-a|\le1).
 \label{eq:circle-counter}
\end{equation}
After rotation, \(a\in(0,1)\).  Put
\begin{equation}
 d=1-a
 \qquad\text{and assume}\qquad
 8^{-n}\le d\le n^{-1/64}.
 \label{eq:circle-range}
\end{equation}
For the random variables \(\lambda\) and \(\xi\), write
\begin{equation*}
 \mu=\E\lambda=\E\xi,\qquad
 \sigma^2=\E|\xi-\mu|^2,\qquad
 V=\E|\xi|^2=\sigma^2+|\mu|^2.
\end{equation*}
By Lemma~\ref{lem:counterexample-support},
\begin{equation}
 |\xi|\le1,\qquad |a-\xi|>1
 \label{eq:circle-support}
\end{equation}
almost surely.

The first lemma supplies the quantitative small parameter used throughout
this section. The polynomial inequalities used in its proof are due to
D\'egot \cite[Lemma~1 and Theorem~4]{Deg14}.

\begin{lem}
\label{lem:circle-preliminary}
Under \eqref{eq:circle-counter} and \eqref{eq:circle-range},
\begin{equation}
 |f(0)|>\frac1{16},\qquad n<|f'(a)|\le16n.
 \label{eq:circle-f0}
\end{equation}
If
$
 L=\log|a-\xi|,
$
then
\begin{equation}
 0<L\le\log2,\qquad \E L<\frac4n,
 \label{eq:circle-log-distance}
\end{equation}
and
\begin{equation}
 V<5n^{-1/64}<10^{-3124}.
 \label{eq:circle-small-V}
\end{equation}
In particular,
\begin{equation}
 |\mu|,\sigma<10^{-1562}.
 \label{eq:circle-small-mu}
\end{equation}
\end{lem}

\begin{proof}
Since
$
 f'(a)=n\prod_{j=1}^{n-1}(a-\xi_j),
$
\eqref{eq:circle-support} gives \(|f'(a)|>n\).
For \(0<t<a\), D\'egot's estimates \cite[Theorem~4, Lemma~1]{Deg14} give
\begin{align}
 |f(t)|
 &\ge\frac{1-\sqrt{1+t^2-at}}n\,|f'(a)|,
 \label{eq:circle-degot-lower}\\
 |f(t)|
 &\le(1+t^2-2t\Re\mu)^{n/2}.
 \label{eq:circle-degot-upper}
\end{align}
The range \eqref{eq:circle-range} implies \(a>0.99\).  If
\(0<t\le0.1\), then
\[
 1-\sqrt{1+t^2-at}
 =\frac{t(a-t)}{1+\sqrt{1-t(a-t)}}>0.44t.
\]
Combining \eqref{eq:circle-degot-lower} and
\eqref{eq:circle-degot-upper} with \(|f'(a)|>n\), taking
logarithms, and using \(\log(1+x)\le x\), yields
\begin{equation}
 \Re\mu\le
 B(t):=\frac t2+\frac{\log(1/(0.44t))}{nt}.
 \label{eq:circle-mean-degot}
\end{equation}
At \(t=n^{-1/3}\), the right side is less than \(0.01\).

Write \(\xi=u+iv\).  From \eqref{eq:circle-support},
\[
 u<\frac{|\xi|^2+a^2-1}{2a}\le\frac a2.
\]
Since \(u\ge-1\) and \(\E u<0.01\), we have
\[
\begin{aligned}
	\E u
	&\ge -\Pp(u\le0.2)+0.2\Pp(u>0.2)\\
	&= -\Pp(u\le0.2)
	+0.2\bigl(1-\Pp(u\le0.2)\bigr)\\
	&=0.2-1.2\Pp(u\le0.2).
\end{aligned}
\]
It follows that
\[
0.2-1.2\Pp(u\le0.2)<0.01,
\]
and therefore
\begin{equation}
	\Pp(u\le0.2)>\frac{19}{120}.
	\label{eq:circle-left-mass}
\end{equation}
For \(0\le x\le1/2\), the support condition \(u<a/2\)
gives
\[
 |a-\xi|^2-|x-\xi|^2
 =(a-x)(a+x-2u)\ge0.
\] Thus
\[
|x-\xi|\le |a-\xi|
\]
for every critical point \(\xi\).
On the event \(u\le0.2\), set
\[
F_a(x):=(a-x)(a+x-0.4)
=a^2-0.4a+0.4x-x^2.
\]
Since \(F_a\) is concave on \([0,1/2]\), its minimum on
this interval is attained at an endpoint.  Using \(a>0.99\),
we obtain
\[
F_a(0)
=a(a-0.4)
\ge0.99(0.99-0.4)
=0.5841
\]
and
\[
F_a(1/2)
=(a-0.5)(a+0.1)
\ge(0.99-0.5)(0.99+0.1)
=0.5341.
\]
Consequently,
\[
F_a(x)\ge0.5341
\qquad (0\le x\le1/2).
\]
Since \(|x-\xi|\le x+|\xi|\le3/2\), it follows that
\[
\frac{|a-\xi|^2}{|x-\xi|^2}
=1+\frac{(a-x)(a+x-2u)}{|x-\xi|^2}
\ge
1+\frac{0.5341}{(3/2)^2}.
\]
Therefore
\[
\log\frac{|a-\xi|}{|x-\xi|}
\ge
\frac12\log\left(1+\frac{0.5341}{(3/2)^2}\right)
>0.1 \qquad\text{on }\{u\le0.2\}.
\]
Using the factorization of \(f'\) and
\eqref{eq:circle-left-mass}, we obtain
\begin{equation}
 |f'(x)|\le e^{-(n-1)/100}|f'(a)|
 \qquad(0\le x\le1/2).
 \label{eq:circle-derivative-decay}
\end{equation}

At \(x=1/2\), \eqref{eq:circle-degot-lower} gives
\[
 |f(1/2)|>\frac{|f'(a)|}{8n}.
\]
Integrating \eqref{eq:circle-derivative-decay} on \([0,1/2]\), and
noting that \(8n e^{-(n-1)/100}<1\), we
find
\[
 |f(0)|
 \ge |f(1/2)|-\frac12e^{-(n-1)/100}|f'(a)|
 >\frac{|f'(a)|}{16n}.
\]
Since \(|f(0)|=\prod_j|\lambda_j|\le1\), this proves
\eqref{eq:circle-f0}.

The factorization of \(f'\) now gives
\[
 (n-1)\E L=\log\frac{|f'(a)|}{n}<\log16,
\]
and hence \eqref{eq:circle-log-distance}.  Moreover,
\[
 e^{2L}=|a-\xi|^2=a^2-2a\Re\xi+|\xi|^2,
\]
so
\begin{equation*}
 V=2a\Re\mu+\E e^{2L}-a^2.
\end{equation*}
The chord joining the endpoints of the convex function \(e^{2x}\)
on \([0,\log2]\) gives
\[
 e^{2L}\le1+\frac3{\log2}L.
\]
Consequently,
\begin{equation*}
 V\le2a\Re\mu+(1-a^2)+\frac{18}{n}.
\end{equation*}
Apply \eqref{eq:circle-mean-degot} with \(t=n^{-1/64}\).
Here
\[
 \log\frac1{0.44t}<\log(3n)<n^{1/16}
 <\frac12n^{31/32},
\]
so \(B(t)<t\).  Since \(d\le t\), it follows that
\[
 V<2t+2d+\frac{18}{n}<5n^{-1/64}.
\]
Since \(n\ge10^{200000}\), the last expression is less than
\(10^{-3124}\).  Equation \eqref{eq:circle-small-mu} follows from
\(|\mu|^2,\sigma^2\le V\).
\end{proof}

The critical points are now close to a fixed circular arc, apart from
a uniformly bounded number of exceptions.  Define
\begin{equation*}
 \mathcal A=\{z\in\C:|z|\le1,\ |1-z|=1\}.
\end{equation*}
The following lemma should be compared with
\cite[pp.~389--390]{Tao22}.
\begin{lem}
\label{lem:circle-projection}
For every zero \(\xi\) of \(f'\), there is \(y\in\mathcal A\) such
that
\begin{equation*}
 |\xi-y|\le6(d+L),
 \qquad L=\log|a-\xi|.
\end{equation*}
\end{lem}

\begin{proof}
Put \(r=|a-\xi|=e^L\) and \(R=|1-\xi|\).  Then
\begin{equation*}
 |R-1|\le|R-r|+(r-1)\le d+2L,
\end{equation*}
because \(e^L-1\le2L\) on \([0,\log2]\).  Radially project \(\xi\),
with center \(1\), to the circle \(|1-z|=1\), obtaining \(y_0\).
Then \(|\xi-y_0|\le d+2L\).

If \(|y_0|\le1\), take \(y=y_0\).  Otherwise write
\(y_0=1+e^{i\phi}\) and take the nearer endpoint of \(\mathcal A\).
The elementary identities
\[
 |y_0|=2|\cos(\phi/2)|,\qquad
 |e^{i\phi}-e^{\pm2\pi i/3}|
 =2\sin\left|\frac{\phi\mp2\pi/3}{2}\right|
\]
show directly that
\[
 \dist(y_0,\mathcal A)
 \le2(|y_0|-1)
 =2\dist(y_0,\overline{\D})
 \le2|\xi-y_0|.
\]
The triangle inequality now gives
\[
 \dist(\xi,\mathcal A)\le3(d+2L)\le6(d+L).
\]
\end{proof}

We fix, once and for all, the parameters
\begin{equation}
\begin{aligned}
 &\Delta=10^{-20},\qquad g=10^{-22},\qquad
 \tau=10^{-30},\qquad K_0=10^{32},\\
 &h=10^{-110},\qquad r_\ast=2.5\cdot10^{-172},\qquad
 d_\ast=2\cdot10^{-172},\qquad \kappa=10^{-400}.
\end{aligned}
\label{eq:circle-parameters}
\end{equation}
Let \(T\) be the multiset of critical points \(t\) for which
\(\dist(t,\mathcal A)>\tau\).  

The path construction in the following lemma is a quantitative version of
the obstacle-avoidance argument used by Tao in the proofs of
\cite[Propositions~5.3 and~5.4]{Tao22}.  In place of the compact exceptional
set arising in Tao's asymptotic argument, we work with finitely many explicit
obstacles and construct paths starting from actual zeros rather than ideal
boundary points.  The lemma also gives an explicit bound for the number of
obstacles, collects all forbidden directions into a single exceptional set
\(H\), and ensures that the resulting admissible paths remain a fixed positive
distance from every obstacle.

\begin{lem}
\label{lem:circle-obstacles}
The multiset \(T\) has fewer than \(K_0\) elements.  There is a
measurable set \(H\subset\mathbb R/(2\pi\mathbb Z)\), invariant under
translation by \(\pi\), with
\begin{equation*}
 |H|<\Delta
\end{equation*}
and the following property.  If \(\theta\notin H\) and \(z\) is
a zero of \(f\) satisfying
\begin{equation}
 |z-e^{i\theta}|\le4\kappa,
 \label{eq:circle-near-boundary-root}
\end{equation}
then \(z\) can be joined to a point \(w_0\), \(|w_0|=0.01\), by a
piecewise smooth path \(\gamma\) of length less than \(1.3\) such that,
with \(s\) denoting arclength from \(z\),
\begin{equation}
	-\frac{\dd}{\dd s}|\gamma(s)|\ge0.8
	\qquad\text{for a.e. }s\in(0,\ell(\gamma)),
	\label{eq:circle-radial-speed}
\end{equation}
and
\begin{equation}
 \dist(\gamma,\supp\xi)>d_\ast.
 \label{eq:circle-path-clearance}
\end{equation}
\end{lem}

\begin{proof}
If \(t\in T\), Lemma~\ref{lem:circle-projection} gives
\(\tau<6(d+L_t)\).  Since \(d<\tau/10\), this implies
\(L_t>\tau/20\).  On the other hand,
\[
 \sum_{j=1}^{n-1}L_j
 =\log\frac{|f'(a)|}{n}<\log16.
\]
Consequently,
\begin{equation*}
 \#T<\frac{20\log16}{\tau}<K_0.
\end{equation*}

For \(u\in\mathbb R\), put
$
 F_u(w)=2\Re w-u|w|^2.
$
Since \(0\in\mathcal A\), every \(t\in T\) is nonzero and
\(|t|>\tau\); set \(u_t=2\Re t/|t|^2\).  Delete every angle satisfying
one of
\begin{align}
 |2\cos\theta-1|&\le2g, \label{eq:circle-delete-transition}\\
 |2\cos\theta-u_t|&\le2h
 &&(t\in T), \label{eq:circle-delete-circle}\\
 |\sin(\theta-\arg t)|
 &\le\frac{2r_\ast}{|t|}
 &&(t\in T). \label{eq:circle-delete-radial}
\end{align}
Enlarge this set by a \(100\kappa\)-neighborhood and take its union
with its translate by \(\pi\).  For \(u,\alpha\in\mathbb R\) and \(r\ge\tau\), the following
one-dimensional estimates hold:
\begin{align*}
	|\{|2\cos\theta-1|\le2g\}|&<20g,\\
	|\{|2\cos\theta-u|\le2h\}|&<40\sqrt h,\\
	|\{|\sin(\theta-\alpha)|\le2r_\ast/r\}|&<32r_\ast/r.
\end{align*}
For completeness, the first set is contained in two intervals
on which \(|\sin\theta|>0.8\).  The mean-value theorem therefore gives
\[
|\{\theta:|2\cos\theta-1|\le2g\}|<20g.
\]
For the second estimate, remove the \(\sqrt h\)-neighborhoods of the
two turning points \(0\) and \(\pi\).  Their total length is at most
\(4\sqrt h\).  On the complement,
\[
|(2\cos\theta)'|
=2|\sin\theta|
\ge\sqrt h.
\]
Applying the mean-value theorem on the two monotonicity intervals of
\(\cos\theta\) gives
\[
|\{\theta:|2\cos\theta-u|\le2h\}|<40\sqrt h.
\]
Finally, if \(0<\eta<1/4\), then
\[
|\{\theta:|\sin(\theta-\alpha)|\le\eta\}|
=4\arcsin\eta\le8\eta.
\]
Taking \(\eta=2r_\ast/r\), which lies in \((0,1/4)\) because
\(r\ge\tau\), yields
\[
|\{\theta:|\sin(\theta-\alpha)|\le2r_\ast/r\}|
<32r_\ast/r.
\]

After the \(\pi\)-symmetrization and thickening, these estimates give
the following explicit bound.  Before symmetrization the deleted set has at most
\(4+8K_0\) interval components; enlarging each component by
\(100\kappa\) adds at most
\(200\kappa(4+8K_0)<2000(K_0+1)\kappa\) to its length.  Therefore
\[
 |H|<
 2\left(20g+40K_0\sqrt h+\frac{32K_0r_\ast}{\tau}
 +2000(K_0+1)\kappa\right)
 <6\cdot10^{-21}<\Delta.
\]

Let \(\theta\notin H\) and let \(z\) satisfy
\eqref{eq:circle-near-boundary-root}.  Put
$
 u_z=2\Re z/|z|^2.
$
The map \(w\mapsto2\Re w/|w|^2\) is \(10\)-Lipschitz on
\(\{|w|\ge0.99\}\); hence
\begin{equation}
 |u_z-2\cos\theta|<50\kappa.
 \label{eq:circle-u-close}
\end{equation}
The deletion \eqref{eq:circle-delete-transition} implies that precisely
one of the following alternatives applies:
\[
 u_z\le0,\qquad
 u_z\ge1+\frac g3,\qquad
 0<u_z\le1-\frac g3.
\]
In the first two alternatives, take the radial segment from \(z\) to
the circle \(|w|=0.01\).  In the third alternative,
\(0<u_z\le1-g/3\) and \(F_{u_z}(z)=0\). In polar coordinates
\(w=re^{i\phi}\), the nonzero part of the level set
\(F_{u_z}(w)=0\) is given by
\[
r=\frac{2\cos\phi}{u_z},
\qquad |\phi|<\frac{\pi}{2}.
\]
Starting at \(z\), move toward \(+\pi/2\) when \(\arg z\ge0\), and
toward \(-\pi/2\) when \(\arg z<0\), until \(r=0.01\).
Along this branch \(r\) decreases monotonically, so the whole arc lies
in \(\overline{\D}\). Since \(F_{u_z}(w)=0\) is the circle centered at
\(1/u_z\) with radius \(1/u_z\), differentiation with respect to
arclength gives
\[
-\frac{\dd r}{\dd s}
=\sqrt{1-\frac{u_z^2r^2}{4}}
\ge\frac{\sqrt3}{2},
\]
where we used \(u_zr\le1\). Thus each chosen path has length less than
\[
\frac{1-0.01}{\sqrt3/2}<1.3
\]
and satisfies \eqref{eq:circle-radial-speed}.

On a circular path,
\[
 |F_1(w)|=|1-u_z||w|^2\ge\frac g3\,10^{-4}.
\]
If the radial path has \(u_z\le0\), then
\[
 |F_1(re^{i\arg z})|
 =r(r-u_z|z|)\ge r^2.
\]
If it has \(u_z\ge1+g/3\), then
\[
 F_1(re^{i\arg z})
 =r(u_z|z|-r)\ge\frac g3\,r|z|
\]
for \(0.01\le r\le|z|\).  Thus the same lower bound holds on every
path.  Since \(F_1=0\) on \(\mathcal A\) and
\(|\nabla F_1|\le4\) in \(\overline{\D}\),
\begin{equation}
 \dist(w,\mathcal A)>5\cdot10^{-28}
 \qquad(w\in\gamma).
 \label{eq:circle-A-clearance}
\end{equation}

For a circular path and \(t\in T\),
\eqref{eq:circle-delete-circle} and \eqref{eq:circle-u-close} give
\[
 |F_{u_z}(t)|
 =|u_z-u_t||t|^2>h\tau^2=10^{-170}.
\]
Here \(|\nabla F_{u_z}|\le4\) on the unit disk, so
\[
 \dist(t,\gamma)>2.5\cdot10^{-171}>d_\ast.
\]
For a radial path, the \(100\kappa\)-thickening and
\eqref{eq:circle-delete-radial}, together with
\(|\arg z-\theta|<5\kappa\), give
\[
\begin{aligned}
 |\sin(\arg z-\arg t)|
 &\ge |\sin(\theta-\arg t)|-|\arg z-\theta|\\
 &>\frac{2r_\ast}{|t|}-5\kappa
  >\frac{r_\ast}{|t|}.
\end{aligned}
\]
Thus the distance from \(t\) to the full line containing the radial
segment, and hence to the segment itself, is greater than
\(r_\ast>d_\ast\).  Finally, if
\(t\in\supp\xi\setminus T\) and \(w\in\gamma\), then
\(\dist(t,\mathcal A)\le\tau\), and hence
\[
|w-t|
\ge \dist(w,\mathcal A)-\dist(t,\mathcal A)
>5\cdot10^{-28}-\tau>d_\ast
\]
by \eqref{eq:circle-A-clearance}. This proves
\eqref{eq:circle-path-clearance}.
\end{proof}

We next show that every boundary direction has a nearby zero of \(f\);
for directions outside \(H\), Lemma~\ref{lem:circle-obstacles} then
applies to that zero. The discrepancy estimate used here is
the Fourier form of the Erd\H{o}s--Tur\'an inequality
\cite[Corollary~1.1]{Mon94}.

\begin{lem}
\label{lem:circle-boundary-roots}
For every \(\theta\in\mathbb R\), there is a zero \(z\) of \(f\) such
that
\begin{equation}
 |z-e^{i\theta}|<4\kappa.
 \label{eq:circle-root-every-angle}
\end{equation}
\end{lem}

\begin{proof}
By \eqref{eq:circle-f0},
\begin{equation}
 \sum_{j=1}^n(1-|\lambda_j|)
 \le\sum_{j=1}^n\log\frac1{|\lambda_j|}
 =\log\frac1{|f(0)|}<\log16<3.
 \label{eq:circle-radial-defect}
\end{equation}
In particular, no \(\lambda_j\) is zero.  Put
\(\omega_j=\lambda_j/|\lambda_j|\), and let \(\nu\) be their empirical
angular measure. We use the convention
\[
\widehat\nu(m)
:=\int_{\partial\D}w^m\,\dd\nu(w)
=\frac1n\sum_{j=1}^n\omega_j^m,
\qquad m\ge1.
\]  Lemma~\ref{lem:moment-comparison} and
\eqref{eq:circle-radial-defect} give
\begin{equation*}
 |\widehat\nu(m)-\E\xi^m|\le\frac{8m}{n}
 \qquad(m\ge1).
\end{equation*}
Moreover,
\[
 |\E\xi|\le\sqrt V,\qquad
 |\E\xi^m|\le V\quad(m\ge2).
\]
In the normalization used here, the Erd\H{o}s--Tur\'an inequality \cite[Corollary~1.1]{Mon94} is
\[
 D(\nu)\le10\left\{M^{-1}
 +\sum_{m=1}^{M}\frac{|\widehat\nu(m)|}{m}\right\}.
\]
With \(M=10^{410}\), it therefore gives
\begin{equation*}
 D(\nu)\le10\left(
 M^{-1}+\sqrt V+V\log M+\frac{8M}{n}\right)
 <\frac{\kappa}{100\pi}.
\end{equation*}
Here \(D(\nu)\) is the supremum, over arcs \(I\), of
\(|\nu(I)-|I|/(2\pi)|\).

The arc centered at \(\theta\) and having angular length \(4\kappa\)
contains more than \(0.6n\kappa\) zero arguments.  By
\eqref{eq:circle-radial-defect}, fewer than \(3/\kappa\) zeros have
radial defect larger than \(\kappa\).  Since
\(n>5\kappa^{-2}\), the centered arc contains a zero with radial
defect at most \(\kappa\).  Its angular displacement is at most
\(2\kappa\), so its Euclidean distance from \(e^{i\theta}\) is less
than \(3\kappa\), proving \eqref{eq:circle-root-every-angle}.
\end{proof}

The geometric preparation is complete.  We next integrate the
differential equation for \(f'\) along the separated paths.

\begin{lem}
\label{lem:circle-endpoint}
Let \(q\) be either \(a\) or a zero \(z\) supplied by
Lemma~\ref{lem:circle-boundary-roots} for an angle outside \(H\).  Then
\begin{equation}
 0=f(0)+\frac{f'(q)(q-\mu)}n(1+\vartheta_q)+E_q,
 \label{eq:circle-endpoint}
\end{equation}
where
\begin{equation}
 |\vartheta_q|\le10^{182}\sigma^2,
 \qquad
 |E_q|<0.03^n.
 \label{eq:circle-endpoint-errors}
\end{equation}
Furthermore,
\begin{equation}
 \left|\log|f(0)|+nU_\xi(q)\right|
 \le3\cdot10^{182}\sigma^2+32(0.06)^n.
 \label{eq:circle-level-at-root}
\end{equation}
\end{lem}

\begin{proof}
For \(q=z\), use the path in
Lemma~\ref{lem:circle-obstacles}.  For \(q=a\), use the positive real
segment from \(a\) to \(0.01\).  Indeed, for \(0.01\le x\le a\),
\[
 |x-\xi|\ge|a-\xi|-|a-x|>x+d\ge0.01.
\]
Along either path, \(|w|\ge0.01\) and
\(\dist(w,\supp\xi)\ge d_\ast\).

Substituting \(x=w-\mu\) and \(y=\xi-\mu\) into the algebraic identity
\[
\frac1{x-y}
=\frac1x+\frac y{x^2}+\frac{y^2}{x^2(x-y)}
\]
and then taking expectations gives
\begin{equation*}
 s_\xi(w)=\frac1{w-\mu}
 +\E\frac{(\xi-\mu)^2}
 {(w-\mu)^2(w-\xi)}.
\end{equation*}
Equations \eqref{eq:circle-small-mu} and
\eqref{eq:circle-path-clearance} imply
\begin{equation*}
 \left|s_\xi(w)-\frac1{w-\mu}\right|
 <10^{178}\sigma^2.
\end{equation*}
If the path is parametrized by arclength from \(q\), then
\eqref{eq:circle-radial-speed} applies when \(q=z\), whereas
\(-\dd|\gamma(s)|/\dd s=1\) when \(q=a\).  Differentiating
\(|\gamma(s)-\mu|\), and using
\[
 \left|
 \frac{\dd}{\dd s}|\gamma(s)-\mu|
 -\frac{\dd}{\dd s}|\gamma(s)|
 \right|
 \le\frac{2|\mu|}{|\gamma(s)|-|\mu|}<0.01,
\]
therefore gives
\[
 -\frac{\dd}{\dd s}|\gamma(s)-\mu|>0.79.
\]
As
$
 0.989<|q-\mu|<1.001,
$
writing \(R(s)=|\gamma(s)-\mu|\), we have
\[
R'(s)<-0.79,
\qquad
R(s)\le R(0)=|q-\mu|<1.001.
\]
Consequently,
\[
\frac{\dd}{\dd s}\log R(s)
=\frac{R'(s)}{R(s)}
<-\frac{0.79}{1.001}<-0.78.
\]
Integration gives
\begin{equation}
	\frac{|\gamma(s)-\mu|}{|q-\mu|}
	\le e^{-0.78s}.
	\label{eq:circle-radial-exponential}
\end{equation}

Along \(\gamma\), both \(f'\) and \(w-\mu\) are nonzero; hence a
continuous logarithm may be chosen along the path. Integrating \(f''/f'=(n-1)s_\xi\) gives
\begin{equation}
 f'(w)=f'(q)
 \left(\frac{w-\mu}{q-\mu}\right)^{n-1}
 \exp\left((n-1)\int_q^w
 \left(s_\xi(v)-\frac1{v-\mu}\right)\dd v\right).
 \label{eq:circle-fprime-transport}
\end{equation}
Let \(C_0=10^{178}\sigma^2\).  For a point at arclength \(s\), the
last exponential differs from \(1\) by at most
\[
 nC_0s\,e^{nC_0s}.
\]
Since \(C_0<10^{-2946}\), define the transport remainder by
\[
R_q:=
\int_q^{w_0}
\left[
f'(w)-f'(q)
\left(\frac{w-\mu}{q-\mu}\right)^{n-1}
\right]\dd w.
\]
Equations \eqref{eq:circle-radial-exponential} and
\eqref{eq:circle-fprime-transport} give
\[
\begin{aligned}
	|R_q|
	&\le |f'(q)|\int_0^{\ell(\gamma)}
	e^{-0.78(n-1)s}nC_0s\,e^{nC_0s}\,\dd s\\
	&\le |f'(q)|\int_0^\infty
	nC_0s\,e^{-0.7ns}\,\dd s\\
	&=\frac{C_0}{0.49n}|f'(q)|
	<\frac{4C_0}{n}|f'(q)|.
\end{aligned}
\]
Thus
\[
\frac{|R_q|}{|f'(q)(q-\mu)|/n}
<10^{182}\sigma^2.
\]
The exponent \(0.7n\) is available because
\[
 -0.78(n-1)+nC_0<-0.7n
\]
for \(n\ge\Nstar\).

The model integral itself is exact:
\[
 \int_q^{w_0}(w-\mu)^{n-1}\dd w
 =\frac{(w_0-\mu)^n-(q-\mu)^n}{n}.
\]
Set
\[
B_q:=\frac{f'(q)(w_0-\mu)^n}
{n(q-\mu)^{n-1}},
\qquad
C_q:=f(w_0)-f(0).
\]
By the arithmetic--geometric mean inequality,
\[
\frac{|f'(q)|}{n}
=\prod_{j=1}^{n-1}|q-\xi_j|
\le
\left(\E|q-\xi|^2\right)^{(n-1)/2}
=
\left(|q-\mu|^2+\sigma^2\right)^{(n-1)/2}.
\]
Consequently,
\[
|B_q|
\le |w_0-\mu|
\left(
|w_0-\mu|
\sqrt{1+\frac{\sigma^2}{|q-\mu|^2}}
\right)^{n-1}
<0.0101(0.011)^{n-1}
<(0.012)^n.
\]
Similarly, for \(|w|\le0.01\),
\[
\begin{aligned}
	|f'(w)|
	&\le n\left(\E|w-\xi|^2\right)^{(n-1)/2}\\
	&=n\left(|w-\mu|^2+\sigma^2\right)^{(n-1)/2}
	<n(0.011)^{n-1}.
\end{aligned}
\]
Integrating on the line segment from \(0\) to \(w_0\), and using
\(n\ge\Nstar\), gives
\[
|C_q|
\le0.01n(0.011)^{n-1}
<(0.012)^n.
\]
Since \(f(q)=0\), the definition of \(R_q\) and the exact model
integral give
\[
f(w_0)
=B_q-\frac{f'(q)(q-\mu)}n+R_q.
\]
As \(f(w_0)=f(0)+C_q\), it follows that
\[
0=f(0)+\frac{f'(q)(q-\mu)}n(1+\vartheta_q)+E_q,
\]
where
\[
\vartheta_q:=
-\frac{nR_q}{f'(q)(q-\mu)},
\qquad
E_q:=C_q-B_q.
\]
Hence
\[
|\vartheta_q|\le10^{182}\sigma^2,
\qquad
|E_q|<2(0.012)^n<(0.03)^n,
\]
which proves \eqref{eq:circle-endpoint} and
\eqref{eq:circle-endpoint-errors}.

It remains to convert the endpoint identity into a potential estimate.
Put \(x=q-\mu\) and \(u=(\xi-\mu)/x\).  Then
\[
 |x|>0.989,\qquad |u|<1.02,\qquad
 |1-u|\ge\frac{d_\ast}{|x|}
\]
for a selected boundary zero; the separation is stronger when \(q=a\).
For these values,
\begin{equation}
 \left|\log|1-u|+\Re u\right|
 \le4\left(\log\frac1{d_\ast}+3\right)|u|^2.
 \label{eq:circle-log-quadratic}
\end{equation}
To verify this inequality, use the Taylor series when \(|u|\le1/2\).
When \(|u|>1/2\), bound its left side by
\(\log(1/d_\ast)+3\) and use \(4|u|^2>1\).
Taking expectations and using \(\E(\xi-\mu)=0\) gives
\begin{align}
	\left|-U_\xi(q)-\log|q-\mu|\right|
	&\le
	4\left(\log\frac1{d_\ast}+3\right)
	\frac{\sigma^2}{|q-\mu|^2} \notag\\
	&\le
	\frac{4(\log(1/d_\ast)+3)}{0.989^2}\sigma^2
	<1700\sigma^2.
	\label{eq:circle-centered-potential}
\end{align}

By \eqref{eq:circle-f0}, division of
\eqref{eq:circle-endpoint} by \(f(0)\) is harmless.  Set
\[
 A_q=\frac{f'(q)(q-\mu)}n,
 \qquad
 \delta_q=\frac{E_q}{f(0)}.
\]
Then
\[
 A_q(1+\vartheta_q)=-f(0)(1+\delta_q),
\]
where
\[
 |\delta_q|<16(0.03)^n\le(0.06)^n.
\]
Both \(\vartheta_q\) and \(\delta_q\) have modulus less than \(1/2\),
and therefore
\[
 \log|f(0)|
 =\log|A_q|+\log|1+\vartheta_q|-\log|1+\delta_q|.
\]
Using \(|\log|1+u||\le2|u|\) for \(|u|\le1/2\), together with
\[
 \log|f'(q)|=\log n-(n-1)U_\xi(q),
\]
and \eqref{eq:circle-centered-potential}, gives
\eqref{eq:circle-level-at-root}.
\end{proof}

The endpoint estimate immediately places the selected zeros on an
effective approximate level set.

\begin{cor}
\label{cor:circle-level-set}
For every selected zero \(z\),
\begin{equation}
 |U_\xi(z)-U_\xi(a)|
 \le\frac{6\cdot10^{182}\sigma^2+64(0.06)^n}{n}.
 \label{eq:circle-level-set}
\end{equation}
\end{cor}

\begin{proof}
Apply \eqref{eq:circle-level-at-root} first to \(z\) and then to \(a\),
and use the triangle inequality.
\end{proof}

The level-set relation will first control the mean and the distance of
\(a\) from the unit circle.  Put
\begin{equation*}
 Q=10^{800}(0.06)^n.
\end{equation*}
The following lemma closely follows the argument of
\cite[Proposition~5.6(i)]{Tao22}, with the relevant
estimates recorded in the quantitative form needed below.
\begin{lem}
\label{lem:circle-mean-control}
One has
\begin{align}
 |\mu|&\le2\cdot10^5\sigma^2+20Q,
 \label{eq:circle-mu-fine}\\
 d&\le2\cdot10^4\sigma^2+3Q.
 \label{eq:circle-d-fine}
\end{align}
\end{lem}

\begin{proof}
For \(w=a\) or for a selected zero \(w=z\), put
\[
 G(w,\xi)=\log\left|1-\frac{\xi}{w}\right|
           +\Re\frac{\xi}{w}.
\]
Here \(|w|>0.989\), \(|\xi/w|<1.02\), and
\[
 \left|1-\frac{\xi}{w}\right|
 =\frac{|w-\xi|}{|w|}
 >\frac{d_\ast}{|w|};
\]
at \(w=a\), the separation is in fact greater than \(1/a\).
The same two-case argument used for
\eqref{eq:circle-log-quadratic}, now with \(u=\xi/w\), gives
\[
 |G(w,\xi)|
 \le4\left(\log\frac1{d_\ast}+3\right)
       \frac{|\xi|^2}{|w|^2}
 <1700|\xi|^2.
\]
Since
\[
 U_\xi(w)=-\log|w|+\Re\frac{\mu}{w}-\E G(w,\xi),
\]
comparison at \(a\) and \(z\) gives
\begin{equation*}
 \left|
 U_\xi(a)-U_\xi(z)
 -\left\{\log\frac{|z|}{a}
 +\Re\left(\left(\frac1a-\frac1z\right)\mu\right)\right\}
 \right|\le4000V.
\end{equation*}
Combining this with Corollary~\ref{cor:circle-level-set} and using
\(|z|\le1\), we obtain
\[
 4000V+
 \frac{6\cdot10^{182}\sigma^2+64(0.06)^n}{n}
 \le
 \left(4000+\frac{6\cdot10^{182}}n\right)V
 +\frac{64\cdot10^{-800}}n\,Q
 <10^4V+Q.
\]
Consequently,
\begin{equation}
 \Re\left(\left(\frac1a-\frac1z\right)\mu\right)
 \ge\log a-E_0,
 \qquad E_0=10^4V+Q.
 \label{eq:circle-halfplanes}
\end{equation}

Choose angles outside \(H\) in the four intervals
\begin{equation*}
\begin{aligned}
 J_1&=[0.98\pi,0.99\pi],&
 J_2&=[1.01\pi,1.02\pi],\\
 J_3&=[0.49\pi,0.50\pi],&
 J_4&=[1.50\pi,1.51\pi].
\end{aligned}
\end{equation*}
This is possible because each interval has length \(0.01\pi\), while
\(|H|<10^{-20}\).  Let \(z_j\) be the corresponding zeros and set
\(c_j=1/a-1/z_j\).

For \(j=1,2\), the imaginary parts of \(c_j\) have opposite signs.
Indeed, up to an error smaller than \(10^{-390}\), the first two
windows give
\[
 -1.001<\Re\frac1{z_j}<-0.997,
\]
while the two imaginary parts of \(1/z_j\) have opposite signs and
absolute values between \(0.031\) and \(0.064\).  Since
\(1<1/a<1.011\), a convex combination of \(c_1,c_2\) is real and
belongs to \([1.9,2.1]\).  Hence
\eqref{eq:circle-halfplanes} and
\(\log(1-d)\ge-d/(1-d)\ge-1.011d\) give
\begin{equation}
 \Re\mu\ge-\frac{1.011d+E_0}{1.9}.
 \label{eq:circle-Re-lower}
\end{equation}
On the other hand, the support inequality
\(\E|a-\xi|^2>1\) gives
\begin{equation}
 \Re\mu\le-d+0.506V.
 \label{eq:circle-Re-upper}
\end{equation}
Comparing \eqref{eq:circle-Re-lower} and
\eqref{eq:circle-Re-upper}, and recalling \(E_0=10^4V+Q\), yields
\begin{align}
 d&\le1.2\cdot10^4V+2Q,
 \label{eq:circle-d-in-V}\\
 |\Re\mu|&\le1.2\cdot10^4V+2Q.
 \label{eq:circle-Re-bound}
\end{align}

For \(j=3,4\), direct trigonometry gives
\[
 |\Re c_j|\le1.2,\qquad
 \Im c_3\ge0.99,\qquad \Im c_4\le-0.99.
\]
For example, on these two windows
\[
 |\cos\theta_j|\le\sin(0.01\pi)<0.032,
 \qquad |\sin\theta_j|\ge\cos(0.01\pi)>0.999,
\]
and the perturbation from \(e^{i\theta_j}\) to \(z_j\) is smaller
than \(4\kappa\).
Applying \eqref{eq:circle-halfplanes} for both signs, followed by
\eqref{eq:circle-d-in-V} and \eqref{eq:circle-Re-bound}, gives
\[
 |\Im\mu|\le3.7\cdot10^4V+6Q.
\]
Thus
\begin{equation}
 |\mu|\le6\cdot10^4V+10Q.
 \label{eq:circle-mu-in-V}
\end{equation}
Since \(V=\sigma^2+|\mu|^2<10^{-3124}\), the term
\(6\cdot10^4|\mu|^2\) can be absorbed into the left side of
\eqref{eq:circle-mu-in-V}, since
\[
 6\cdot10^4|\mu|<10^{-1557}<\frac12.
\]
It follows first that
\[
 |\mu|\le1.2\cdot10^5\sigma^2+20Q
 <2\cdot10^5\sigma^2+20Q,
\]
which is \eqref{eq:circle-mu-fine}.  Moreover,
\[
 1.2\cdot10^4|\mu|^2
 \le1.2\cdot10^4\,10^{-1562}
       \bigl(2\cdot10^5\sigma^2+20Q\bigr)
 <8\cdot10^3\sigma^2+Q.
\]
Substitution in \eqref{eq:circle-d-in-V} now gives
\eqref{eq:circle-d-fine}.
\end{proof}

We now compare the nearby zero with the exact boundary point.  The
one-sided form of this comparison is essential: it retains the favorable
term \(-\log|z|\). The next lemma is a minor quantitative refinement of
\cite[Proposition~5.6(ii)]{Tao22}, obtained by making
the Taylor remainder explicit.

\begin{lem}
\label{lem:circle-good-angle}
For every \(\theta\notin H\),
\begin{equation}
 U_\xi(a)-U_\xi(e^{i\theta})
 \ge-10^{-20}V-Q.
 \label{eq:circle-good-angle}
\end{equation}
\end{lem}

\begin{proof}
Let \(z\) be the selected zero near \(e^{i\theta}\), and define
\[
 G(w,\xi)=
 \log\left|1-\frac{\xi}{w}\right|+\Re\frac{\xi}{w}.
\]
Direct differentiation gives the exact identity
\begin{equation}
 |\nabla_wG(w,\xi)|
 =\frac{|\xi|^2}{|w|^2|w-\xi|}.
 \label{eq:circle-G-gradient}
\end{equation}
The segment from \(z\) to \(e^{i\theta}\) remains at distance
greater than \(d_\ast/2\) from every critical point, because its length
is less than \(4\kappa\).  Since
\[
 U_\xi(w)=-\log|w|+\Re\frac{\mu}{w}-\E G(w,\xi)
\]
and \(-\log|z|\ge0\), the mean-value theorem and
\eqref{eq:circle-G-gradient} give
\begin{equation*}
 U_\xi(z)-U_\xi(e^{i\theta})
 \ge-5\kappa|\mu|-\frac{12\kappa}{d_\ast}V.
\end{equation*}
Combine this with \eqref{eq:circle-level-set} and
\eqref{eq:circle-mu-fine}.  The numerical values in
\eqref{eq:circle-parameters}, together with \(n\ge\Nstar\), make the
sum of the \(V\)-terms smaller than \(10^{-20}V\), while all
exponential terms are bounded by \(Q\).  This proves
\eqref{eq:circle-good-angle}.
\end{proof}

The proof now closes by comparing two estimates for the second moment.
Let \(dm=\dd\theta/(2\pi)\); the set \(H\) is already
\(\pi\)-translation-invariant. The following Fourier estimate is a quantitative version of the
second-order closure argument in
\cite[Proposition~5.6(iii) and (5.23)]{Tao22}.

\begin{lem}
\label{lem:circle-Fourier}
One has
\begin{align}
 \E L&\le1.01\cdot10^{-10}V+1.01Q,
 \label{eq:circle-EL-small}\\
 \left|\Re\E\xi^2\right|
 &\le8.01\cdot10^{-10}V+8Q.
 \label{eq:circle-second-small}
\end{align}
\end{lem}

\begin{proof}
For \(0<r<1\), expand the potential of \(r\xi\) on the unit
circle and then let \(r\uparrow1\). Since the law of \(\xi\) has
finite support, the logarithmic kernels converge in \(L^2(dm)\).
Thus all Fourier identities below hold in \(L^2(dm)\), and hence
almost everywhere; the finitely many possible singular angles have
measure zero. In this sense,
\begin{equation*}
 U_\xi(e^{i\theta})
 =\Re(\mu e^{-i\theta})+\mathcal R(\theta),
\end{equation*}
where Parseval's identity gives
\[
 \|\mathcal R\|_2^2
 =\frac12\sum_{k=2}^{\infty}
 \frac{|\E\xi^k|^2}{k^2}
 \le\frac{V^2}{2}\sum_{k=2}^{\infty}\frac1{k^2}.
\]
Thus
\begin{equation}
 \|\mathcal R\|_2<0.568V.
 \label{eq:circle-R-L2}
\end{equation}
Jensen's formula and Fourier inversion give
\begin{align}
 \int U_\xi(e^{i\theta})\,dm&=0,
 \label{eq:circle-mean-zero}\\
 \int\cos(2\theta)U_\xi(e^{i\theta})\,dm
 &=\frac14\Re\E\xi^2.
 \label{eq:circle-second-Fourier}
\end{align}

Because \(H+\pi=H\), the linear term
\(\Re(\mu e^{-i\theta})\) integrates to zero over \(H\) against
each of the \(\pi\)-periodic weights
\[
 1,\qquad1+\cos2\theta,\qquad1-\cos2\theta.
\]
Also, \(m(H)<\Delta\).  Integrating
\eqref{eq:circle-good-angle} over the complement of \(H\), using
\(U_\xi(a)=-\E L\), \eqref{eq:circle-mean-zero}, and
\eqref{eq:circle-R-L2}, gives
\[
 (1-m(H))\E L
 \le0.568\sqrt\Delta\,V+10^{-20}V+Q.
\]
This proves \eqref{eq:circle-EL-small}.

Put
\[
\varepsilon:=10^{-20}V+Q,
\qquad
W_\pm(\theta):=1\pm\cos(2\theta).
\]
By \eqref{eq:circle-good-angle},
\[
\E L+U_\xi(e^{i\theta})\le\varepsilon
\qquad(\theta\notin H).
\]
Since \(\E L\ge0\) and \(W_\pm\ge0\), integration over \(H^c\)
gives
\[
\int_{H^c}W_\pm(\theta)U_\xi(e^{i\theta})\,dm
\le
\varepsilon\int_{H^c}W_\pm(\theta)\,dm.
\]
By \eqref{eq:circle-mean-zero} and
\eqref{eq:circle-second-Fourier},
\[
\int W_\pm(\theta)U_\xi(e^{i\theta})\,dm
=\pm\frac14\Re\E\xi^2.
\]
Because \(H+\pi=H\), the linear term
\(\Re(\mu e^{-i\theta})\) integrates to zero over \(H\) against
\(W_\pm\).  Therefore
\[
\pm\frac14\Re\E\xi^2
\le
\varepsilon
+\left|\int_H W_\pm(\theta)\mathcal R(\theta)\,dm\right|.
\]
Furthermore,
\[
\|W_\pm\mathbf 1_H\|_2
\le2\sqrt{m(H)}
<2\sqrt\Delta.
\]
Thus, by \eqref{eq:circle-R-L2},
\[
\frac14\left|\Re\E\xi^2\right|
\le
2(0.568)\sqrt\Delta\,V+\varepsilon.
\]
The numerical values of \(\Delta\) and \(\varepsilon\) now imply
\[
\left|\Re\E\xi^2\right|
\le8.01\cdot10^{-10}V+8Q,
\]
which is \eqref{eq:circle-second-small}.
\end{proof}

The geometry of \(\mathcal A\) supplies the incompatible estimate.
The next lemma follows Tao's arc-projection and negative second-moment
argument in \cite[pp.~389--390]{Tao22}, with a minor adjustment to the
absorption step based on Young's inequality and with all constants
made explicit.

\begin{lem}
\label{lem:circle-geometric}
One has
\begin{equation}
 \Re\E\xi^2
 \le\left(-\frac14+4.04\cdot10^{-8}\right)V
 +576d^2+404Q.
 \label{eq:circle-geometric-second}
\end{equation}
\end{lem}

\begin{proof}
Choose \(y\in\mathcal A\) as in
Lemma~\ref{lem:circle-projection}.  If \(y=0\), then
\[
\Re y^2=0=-\frac12|y|^2.
\]
Otherwise write \(y=re^{i\phi}\).  Since \(y\in\mathcal A\),
we have \(r=2\cos\phi\), and \(0<r\le1\) implies
\[
|\phi|\in[\pi/3,\pi/2].
\]
Consequently, in all cases,
\begin{equation*}
	\Re y^2\le-\frac12|y|^2.
\end{equation*}
Write \(\xi=y+e\).  Using \(|\xi|\le |y|+|e|\), we obtain
\begin{align*}
	\Re\xi^2+\frac14|\xi|^2
	&\le-\frac12|y|^2+2|y||e|+|e|^2
	+\frac14(|y|+|e|)^2\\
	&=-\frac14\bigl(|y|-5|e|\bigr)^2
	+\frac{15}{2}|e|^2\\
	&\le\frac{15}{2}|e|^2.
\end{align*}
Since \(|e|\le6(d+L)\), it follows that
\begin{equation*}
	\Re\xi^2
	\le-\frac14|\xi|^2+270(d+L)^2
	\le-\frac14|\xi|^2+576(d^2+L^2).
\end{equation*}
Now \(L^2\le(\log2)L\).  Taking expectations and applying
\eqref{eq:circle-EL-small} gives
\eqref{eq:circle-geometric-second}.
\end{proof}

Combining the preceding quantitative estimates yields the following
effective version of \cite[Theorem~5.1]{Tao22}.

\begin{prop}
\label{prop:circle-effective}
No polynomial satisfying \eqref{eq:circle-counter} and
\eqref{eq:circle-range} exists when \(n\ge\Nstar\).
\end{prop}

\begin{proof}
Suppose first that
\[
 \sigma^2\le10^{1000}(0.06)^n.
\]
\eqref{eq:circle-d-fine} gives
\[
 d<3\cdot10^{1004}(0.06)^n<8^{-n},
\]
contrary to \eqref{eq:circle-range}.

It remains to consider
$
 \sigma^2>10^{1000}(0.06)^n.
$
 Then
\[
\eta:=\frac{Q}{\sigma^2}<10^{-200}.
\]
By \eqref{eq:circle-mu-fine},
\[
|\mu|
\le(2\cdot10^5+20\eta)\sigma^2
<2.01\cdot10^5\sigma^2.
\]
Since \(\sigma^2\le V\) and \eqref{eq:circle-small-V} gives
\(V<10^{-3124}\), it follows that
\[
\frac{V}{\sigma^2}
=1+\frac{|\mu|^2}{\sigma^2}
\le1+(2.01\cdot10^5)^2\sigma^2
<1.001.
\]
Likewise, \eqref{eq:circle-d-fine} gives
\[
d\le(2\cdot10^4+3\eta)\sigma^2
<2.01\cdot10^4\sigma^2.
\]
Hence, using \(V\ge\sigma^2\) and again
\eqref{eq:circle-small-V},
\[
\frac{576d^2}{V}
\le576(2.01\cdot10^4)^2\sigma^2
<10^{-1400}.
\]

The geometric estimate \eqref{eq:circle-geometric-second} therefore
implies
\[
 \Re\E\xi^2<-0.2499999V.
\]
The Fourier estimate \eqref{eq:circle-second-small}, on the other hand,
implies
\[
 \Re\E\xi^2>-8.02\cdot10^{-10}V.
\]
These inequalities are incompatible.
\end{proof}

The compactness argument near the unit circle has therefore been replaced
by finite obstacle removal and explicit estimates.  We now combine the
two endpoint theorems with the quantitative results already in the
literature.

\section{Assembly of the ranges}
\label{sec:assembly}

We first record that the explicit constants used at the two endpoints
are compatible:
\begin{equation}
 \Norg<10^{368}<\Nstar.
 \label{eq:threshold-comparison}
\end{equation}
Indeed, the base in \eqref{eq:org-threshold} is less than \(10^{23}\).
The following proof now exhausts all possible positions of the
distinguished zero.

\begin{proof}[Proof of Theorem~\ref{thm:main}]
Suppose to the contrary that \(n\ge\Nstar\) and that \(f\) is a
counterexample.  Multiplication by a scalar and rotation preserve the
problem, so \(f\) may be taken monic and its distinguished zero may be
taken to be \(a\in[0,1]\).  Assume that \(f'\) has no zero in
\(\overline{D(a,1)}\), and put \(d=1-a\).

If \(a=0\), the Gauss--Lucas theorem places every critical point in
\(\overline{\D}=\overline{D(a,1)}\), a contradiction.  If
\(0<a\le n^{-1/8}\), then \eqref{eq:threshold-comparison} and
Theorem~\ref{thm:org-effective} give the same contradiction.

It remains to suppose that \(a>n^{-1/8}\).  If
\(d\ge n^{-1/64}\), then
\[
 na^7d^4
 \ge n^{1-7/8-4/64}
 =n^{1/16}>20800.
\]
Equivalently,
\[
 n>\frac{20800}{a^7(1-a)^4},
\]
so Chalebgwa's theorem
\cite[Theorem~1.1]{Cha20} applies and rules out this case.

We are left with \(d<n^{-1/64}\).  If \(d\ge8^{-n}\),
Proposition~\ref{prop:circle-effective} rules out the counterexample.
Finally, suppose that \(0\le d<8^{-n}\).  For every \(n\ge64\),
\[
 2n^9\le2^n,
\]
because it holds at \(n=64\) and the ratio \(2^n/(2n^9)\) is
increasing thereafter.  Hence
\[
 8^{-n}\le\frac1{2n^94^n}.
\]
Thus
\[
 a>1-\frac1{2n^94^n},
\]
and Chijiwa's explicit boundary theorem
\cite[Main theorem]{Chi11} applies.  More precisely, it
supplies a zero \(\zeta\) of \(f'\) for which
\[
 |\zeta-a|\le1-c_n(1-a),
\]
where Chijiwa's explicit congruence-dependent coefficient \(c_n\) is
nonnegative for \(n\ge9\).  Hence \(|\zeta-a|\le1\).  Every possible
value of \(a\) has led to a contradiction.
\end{proof}

\begin{rem}
\label{rem:not-optimal}
The theorem identifies a valid universal threshold, not the smallest
one.  The numerical value is dominated by the decision to separate all
finite obstacles using fixed scales and by the high Fourier cutoff
\(M=10^{410}\).  None of the inequalities in the proof is close to
equality.
\end{rem}

\appendix

\section{Numerical audit}
\label{sec:numerical-audit}

For completeness, Table~\ref{tab:numerical-margins} gathers the
numerical implications of \(n\ge10^{200000}\) that were used in the
proof. Each bound in the table follows by taking base-ten logarithms;
the estimates are intentionally weakened to leave margins of several
orders of magnitude.

\begin{table}[ht]
\centering
\caption{Numerical margins in the effective proof}
\label{tab:numerical-margins}
\small
\begin{tabular}{@{}lll@{}}
\toprule
Quantity & Bound used & Consequence\\
\midrule
\(n^{-1/64}\) & \(\le10^{-3125}\) &
\(5n^{-1/64}<10^{-3124}\)\\
\(n^{1/16}\) & \(\ge10^{12500}\) &
\(n^{1/16}>20800\)\\
\(\Norg\) & \(<10^{368}\) &
\(\Norg<\Nstar\)\\
\(D(\nu)\) &
\(<10(10^{-410}+10^{-1562}+10^{-3119}+10^{-199589})\) &
\(<\kappa/(100\pi)\)\\
\(|H|\) &
\(<6\cdot10^{-21}\) &
\(<\Delta\)\\
\(\kappa/d_\ast\) & \(=5\cdot10^{-229}\) &
\(<10^{-20}\)\\
\(Q/\sigma^2\) in the second case & \(<10^{-200}\) &
\(V<1.001\sigma^2\)\\
\bottomrule
\end{tabular}
\end{table}

We spell out three checks that involve sums rather than a single
monomial comparison.  First, the exceptional-angle construction gives
\begin{align*}
 |H|
 &<2\left(
 20\cdot10^{-22}
 +40\cdot10^{32}\cdot10^{-55}
 +\frac{32\cdot10^{32}\cdot2.5\cdot10^{-172}}{10^{-30}}
 +2000(10^{32}+1)10^{-400}\right)\\
 &<6\cdot10^{-21}.
\end{align*}
Second, in Lemma~\ref{lem:circle-endpoint}, the two absolute errors
arising at the inner endpoint are each bounded by \(0.012^n\); hence
\[
 2(0.012)^n<0.03^n.
\]
Third, in the small-variance alternative of
Proposition~\ref{prop:circle-effective},
\[
 \frac{3\cdot10^{1004}(0.06)^n}{8^{-n}}
 =3\cdot10^{1004}(0.48)^n<1.
\]
These checks also show why a decimal used as an upper bound must not be
rounded downward; every displayed decimal used as an upper bound has
been rounded upward.

\end{document}